\documentclass[11pt]{amsart}
\usepackage[utf8]{inputenc}
\usepackage{amsmath,amssymb,amsthm,amsfonts,amscd}
\usepackage{mathtools}
\usepackage{algpseudocode}
\usepackage{algorithm}
\usepackage{tikz}
\usetikzlibrary{decorations.markings}

\tikzstyle{interrupt}=[
    postaction={
        decorate,
        decoration={markings,
                    mark= at position 0.5 
                          with
                          {
                            \fill[white] (-0.1,-0.1) rectangle (0.1,0.1);
                            \draw (-0.1,0.1) -- (-0.1,-0.1)
                                  (0.1,0.1) -- (0.1,-0.1);
                          }
                    }
                }
]
\usepackage{tikz-cd}
\usepackage{pst-node}
\usepackage[margin=3cm]{geometry}
\usepackage{enumerate}
\usepackage{hyperref}
\usepackage[capitalise]{cleveref}
\usepackage{graphicx}
\usepackage{float}
\usepackage{enumerate}
\usepackage{autonum}
\makeatletter
\newtheorem*{rep@theorem}{\rep@title}
\newcommand{\newreptheorem}[2]{%
\newenvironment{rep#1}[1]{%
 \def\rep@title{#2 \ref{##1}}%
 \begin{rep@theorem}}%
 {\end{rep@theorem}}}
\makeatother

\newtheorem{thm}{Theorem}[section]
\newtheorem{Lemma}[thm]{Lemma}
\newtheorem{Proposition}[thm]{Proposition}
\newtheorem{Corollary}[thm]{Corollary}

\newtheorem*{disclaimer}{Disclaimer}

\theoremstyle{definition}
\newtheorem{Definition}[thm]{Definition}

\newtheorem{Remark}[thm]{Remark}
\newtheorem{Algorithm}[thm]{Algorithm}
\newtheorem{Notation}[thm]{Notation}
\newtheorem{Example}[thm]{Example}
\newtheorem{Construction}[thm]{Construction}

\allowdisplaybreaks

\title[A monodromy graph approach to double Hurwitz numbers]{A monodromy graph approach to the piecewise polynomiality of simple, monotone and Grothendieck dessins d'enfants double Hurwitz numbers}
\author{Marvin Anas Hahn}
\address{Mathematisches Institut, Universit\"at T\"ubingen, Auf der Morgenstelle 10, 72076 T\"ubingen, Germany}
\email{marvinanashahn@gmail.com}

\begin{document}
\begin{abstract}
Hurwitz numbers count genus $g$, degree $d$ covers of the complex projective line with fixed branched locus and fixed ramification data. An equivalent description is given by factorisations in the symmetric group. Simple double Hurwitz numbers are a class of Hurwitz-type counts of specific interest. In recent years a related counting problem in the context of random matrix theory was introduced as so-called monotone Hurwitz numbers. These can be viewed as a desymmetrised version of the Hurwitz-problem. A combinatorial interpolation between simple and monotone double Hurwitz numbers was introduced as mixed double Hurwitz numbers and it was proved that these objects are piecewise polynomial in a certain sense. Moreover, the notion of strictly monotone Hurwitz numbers has risen in interest as it is equivalent to a certain Grothendieck dessins d'enfant count. In this paper, we introduce a combinatorial interpolation between simple, monotone and strictly monotone double Hurwitz numbers as \textit{triply interpolated Hurwitz numbers}. Our aim is twofold: Using a connection between triply interpolated Hurwitz numbers and tropical covers in terms of so-called monodromy graphs, we give algorithms to compute the polynomials for triply interpolated Hurwitz numbers in all genera using Erhart theory. We further use this approach to study the wall-crossing behaviour of triply interpolated Hurwitz numbers in genus $0$ in terms of related Hurwitz-type counts. All those results specialise to the extremal cases of simple, monotone and Grothendieck dessins d'enfants Hurwitz numbers.
\end{abstract}
\maketitle
\section{Introduction}
Hurwitz numbers are important enumerative invariants connecting various areas of mathematics, such as algebraic gometry, combinatorics, representation theory, operator theory, tropical geometry and many more. Introduced by Adolf Hurwitz in 1891, they were used to study the moduli space of genus $g$ curves \cite{Hurwitz}. There are several equivalent definitions of Hurwitz numbers. The one originally introduced by Hurwitz is of a topological nature: It counts the number of branched genus $g$, degree $d$ covers of $\mathbb{P}^1$ with fixed ramification data over $n$ fixed points. Another description, which is essentially due to Hurwitz as well, interprets Hurwitz numbers as the enumeration of certain factorisations in the symmetric group.\\
Hurwitz numbers have been a focal point of study in the last two decades when their relationship to Gromov-Witten theory and mathematical physics was uncovered. The cases of single and double Hurwitz numbers have proven to be of specific interest. Single Hurwitz numbers count covers of $\mathbb{P}^1$ with ramification profile $\mu$ at $0\in \mathbb{P}^1$ (where $\mu$ is a partition of the degree) and simple ramification data at $m$ arbitrary fixed points (where $m$ is determined by the Riemann-Hurwitz formula). Simple double Hurwitz numbers count covers of $\mathbb{P}^1$ with ramification profile $\mu$ at $0$ and $\nu$ at $\infty$ and simple ramification at $m$ arbitrary fixed points (where again $m$ is given by the Riemann-Hurwitz formula). A remarkable result relating single Hurwitz numbers to intersection products is the celebrated ELSV formula \cite{ekedahl2001hurwitz}. An immediate consequence is that -- up to a combinatorial factor -- single Hurwitz numbers behave polynomially in the entries of the partition $\mu$ specifying the ramification data over $0$.\\
Many properties of single Hurwitz numbers translate to similar properties of simple double Hurwitz numbers, e.g. it was proved in \cite{GJV} that simple double Hurwitz numbers behave piecewise polynomially in the entries of the partitions $\mu$ and $\nu$ specifying the ramification data over $0$ and $\infty$. The natural question whether there is an ELSV-type formula for simple double Hurwitz numbers remains an active field of research. ELSV-type formulas are closely related to the Chekhov-Eynard-Orantin topological recursion (\cite{chekhov2006hermitian}, \cite{eynard2009laplace}), which is a way of associating a recursion involving differential forms to a spectral curve.
There are spectral curves with which differential forms can be produced that can be viewed as generating functions for single Hurwitz numbers -- one can say that single Hurwitz numbers satisfy the topological recursion. The question whether simple double Hurwitz number satisfy the topological recursion in some sense is also an active field of research (\cite{guay2014generating}, \cite{alexandrov2016weighted}). There are several cases where an ELSV-type formula was derived from topological recursion (\cite{dunin2015polynomiality}, \cite{alexandrov2015ramifications}). By proving that an enumerative problem satisfies the topological recursion, one often makes use of the (quasi-)polynomiality of this problem (\cite{kramer2016quasi}). In connection with topological recursion and ELSV-formulas, the piecewise polynomial structure of simple double Hurwitz numbers gains relevance.\\
The chamber behaviour induced by the piecewise polynomial structure of double Hurwitz numbers was first studied in \cite{SSV} and by using degeneration techniques, wall-crossing formulas for genus $0$ were given. The problem of understanding the chamber behaviour for higher genera remained unanswered until wall-crossing formulas for arbitrary genus were given in \cite{CJMa} and \cite{johnson2015}. In \cite{CJMa}, these formulas were proved using so-called monodromy graphs which essentially express double Hurwitz numbers in terms of covers as they appear in tropical geometry. This description was developed in \cite{CJM} by giving a graph theoretic interpretation of factorisations in the symmetric group.\\
There are several variants on the definition of Hurwitz numbers yielding so-called \textit{Hurwitz-type counts}. Two of the more important Hurwitz-type counts are so-called \textit{monotone} and \textit{strictly monotone Hurwitz numbers}. Monotone Hurwitz numbers were introduced in \cite{goulden2014monotone}. They are defined in the symmetric group setting and show up in the computation of the HCIZ integral. A tropical interpretation of certain monotone Hurwitz numbers in the flavour of \cite{CJM} was developed in \cite{do2015monotone}, where a conjecture on the topological recursion for single monotone Hurwitz numbers was stated (for further literatue on monotone Hurwitz numbers and topological recursion, see e.g. \cite{alexandrov2015ramifications}, \cite{do2014topological}).\\
Further, the notion of strictly monotone Hurwitz numbers have gained attention as is equivalent to counting certain Grothendieck dessins d'enfants \cite{alexandrov2015ramifications}. As for monotone Hurwitz numbers, topological recursion was proved for many cases of strictly monotone double Hurwitz numbers \cite{norbury2009string,do2013quantum,dumitrescu2013spectral,kazarian2015virasoro}.\\
A combinatorial interpolation between double Hurwitz numbers and monotone double Hurwitz numbers was studied in \cite{zbMATH06586291} as \textit{mixed Hurwitz numbers}. It was proved that those new enumerative objects satisfy a piecewise polynomiality result as well, thus proving piecewise polynomiality for monotone double Hurwitz numbers and giving a new proof for the piecewise polynomiality of double Hurwitz numbers. This was done in terms of character theory.\\
It is natural to ask for wall-crossing formulas for mixed Hurwitz numbers. In this paper we combine several approaches to Hurwitz numbers to answer this question. We begin by introducing a combinatorial interpolation between simple, monotone and strictly monotone Hurwitz numbers called \textit{triply interpolated Hurwitz numbers}. We further find a tropical interpretation of triply interpolated Hurwitz numbers in the flavour of \cite{do2015monotone}. Using this description, we develop an algorithm to compute the polynomials in each chamber (see \cref{alg:1} for the case of genus $0$ and \cref{alg:2} for the case of arbitrary genus, which involves Erhart theory, more precisely the integration over lattice points in a polytope. Moreover, we generalise this method to give an algorithm computing the polynomials in each chamber by certain polytope-computations for arbitrary genus.\\
Finally, we introduce a Hurwitz-type counting problem generalising triply interpolated Hurwitz numbers in genus $0$, which is accessible by our algorithms as well. We give recursive wall-crossing formulas for this generalised counting problem, which in particular implies wall-crossing formulas for triply interpolated Hurwitz numbers.\\
In \cref{sec:pre}, we recall some of the basic facts about Hurwitz numbers and outline the previous polynomiality results. In \cref{sec:gr}, we describe triply interpolated Hurwitz numbers in terms of monodromy graphs. We use this description in \cref{sec:pie} to compute the polynomials in each chamber. We begin this discussion for genus $0$ in \cref{sec:pi0} and generalise the method for higher genus in \cref{sec:pi1}. In \cref{sec:ch}, we define a generalisation of triply interpolated Hurwitz numbers in genus $0$ and explain how this new enumerative problem yields a recursive wall-crossing formula, which is in particular valid for triply interpolated Hurwitz numbers.\vspace{\baselineskip}\\

\textbf{Acknowledgements.} I am indebted to my advisor Hannah Markwig for many helpful suggestions, her guidance and extensive proof-reading throughout the preparation of this paper. Moreover, the author thanks Maksim Karev, Reinier Kramer and Danilo Lewanski for helpful comments and discusions. The author gratefully acknowledges partial support by DFG SFB-TRR 195 "Symbolic tool in mathematics and their applications", project A 14 "Random matrices and Hurwitz numbers" (INST 248/238-1).

\section{Preliminaries}
\label{sec:pre}
We begin by introducing the basic notions of Hurwitz numbers.

\begin{Definition}
\label{def:hurwitz}
 Let $d$ be a positive integers, $\mu,\nu$ two ordered partitions of $d$ and let $g$ be a non-negative integer. Moreover, let $q_1,\dots,q_b$ be points in $\mathbb{P}^1$, where $b=2g-2+\ell(\mu)+\ell(\nu)$. We define a Hurwitz cover of type $(g,\mu,\nu)$ to be a holomorphic map $\pi:C\to\mathbb{P}^1$, such that:
  \begin{enumerate}[(1)]
  \item $C$ is a connected genus $g$ curve,
  \item $\pi$ is a degree $d$ map, with ramification profile $\mu$ over $0$, $\nu$ over $\infty$ and $(2,1,\dots,1)$ over $q_i$ for all $i=1,\dots,b$,
  \item $\pi$ is unramified everywhere else,
  \item the pre-images of $0$ and $\infty$ are labeled, such that the point labeled $i$ in $\pi^{-1}(0)$ (respectively $\pi^{-1}(\infty)$) has ramification index $\mu_i$ (respectively $\nu_i)$.  
 \end{enumerate}
We define an isomorphism between two covers $\pi:C_1\to\mathbb{P}^1$ and $\pi':C_2\to\mathbb{P}^1$ to be a homeomorphism $\phi:C_1\to C_2$ respecting the labels, such that the following diagram commutes:
\begin{equation}
\begin{tikzcd}
C_1 \arrow{r}{\phi} \arrow{d}{\pi} & C_2 \arrow{d}{\pi'} \\
\mathbb{P}^1 \arrow{r}{\mathrm{id}} & \mathbb{P}^1
\end{tikzcd}
\end{equation}
Then we define the simple double Hurwitz numbers as follows: \begin{equation}h_{b;\mu,\nu}=\sum \frac{1}{|\mathrm{Aut}(\pi)|},\end{equation} where the sum goes over all isomorphism classes of Hurwitz covers of type $(g,\mu,\nu)$. This number does not depend on the position of the $q_i$. The degree is implicit in the notation $h_{b;\mu,\nu}$, as $d=\sum \mu_i=\sum \nu_j$. The genus $g$ of simple branch points is determined by the Riemann-Hurwitz formula, so $b=2g-2+\ell(\mu)+\ell(\nu)$ as above.\par
When we drop the condition on $C$ to be a \textit{connected} curve and the labels in condition (4), we obtain the notion of \textit{disconnected simple double Hurwitz numbers} $h^{\circ}_{b;\mu,\nu}$.
\end{Definition}

\begin{Remark}
The notions of simple double Hurwitz numbers and disconnected simple double Hurwitz numbers determine each other by the inclusion-exclusion principle.
\end{Remark}

For $\sigma\in S_d$, we denote its cycle type by $\mathcal{C}(\sigma)\vdash d$. We define the following factorisation counting problem in the symmetric group.
\begin{Definition}
\label{def:hursym}
Let $d,g,\mu,\nu$ be as in \cref{def:hurwitz}. We call $\left(\sigma_1,\tau_1,\dots,\tau_b,\sigma_2\right)$ a factorisation of type $(g,\mu,\nu)$, if:
\begin{enumerate}
\item $\sigma_1,\ \sigma_2,\ \tau_i\in\mathcal{S}_d$,
\item $\sigma_2\cdot\tau_b\cdot\dots\cdot\tau_1\cdot\sigma_1=\mathrm{id}$,
\item $b=2g-2+\ell(\mu)+\ell(\nu)$
\item $\mathcal{C}(\sigma_1)=\mu,\ \mathcal{C}(\sigma_2)=\nu$ and $\mathcal{C}(\tau_i)=(2,1,\dots,1)$,
\item the group generated by $\left(\sigma_1,\tau_1,\dots,\tau_b,\sigma_2\right)$ acts transitively on $\{1,\dots,d\}$,
\item the disjoint cycles of $\sigma_1$ and $\sigma_2$ are labeled, such that the cycle $i$ has length $\mu_i$.
\end{enumerate}
We denote the set of all factorisations of type $(g,\mu,\nu)$ by $\mathcal{F}(g,\mu,\nu)$. 
\end{Definition}
A well-known fact is the following theorem, which is essentially due to Hurwitz.
\begin{thm}
\label{thm:hursym}
Let $g,\mu,\nu$ as in the previous definition, then
\begin{equation}h_{b;\mu,\nu}=\frac{1}{d!}\left|\mathcal{F}(g,\mu,\nu)\right|.\end{equation}
\end{thm}

\begin{Remark}
We can modify \cref{thm:hursym} for disconnected simple double Hurwitz numbers by dropping the transitivity condition in \cref{def:hursym}.
\end{Remark}

As proved in \cite{goulden2014monotone} monotone double Hurwitz numbers appear as the coefficients of the HCIZ-integral. They can be defined as counts of factorisations as in \cref{def:hursym} by imposing an additional condition on the transpositions:
\begin{Definition}
Let $k$ be a non-negative integer. We call $\left(\sigma_1,\tau_1,\dots,\tau_b,\sigma_2\right)$ a \textit{monotone factorisation} of type $(g,\mu,\nu)$, if it is a factorisation of type $(g,\mu,\nu)$, such that:
\begin{enumerate}
\item [(7a)] If $\tau_i=(r_i\ s_i)$ with $r_i<s_i$, we have $s_i\le s_{i+1}$ for all $i=1,\dots,b-1$. 
\end{enumerate}
Let $\vec{\mathcal{F}}(g,\mu,\nu)$ be the set of all monotone factorisations of type $(g,\mu,\nu)$. Then we define the \textit{monotone double Hurwitz number} to be:
\begin{equation}h^{\le}_{b;\mu,\nu}=\frac{1}{d!}\left|\vec{\mathcal{F}}(g,\mu,\nu)\right|.\end{equation}
\end{Definition}

\begin{Remark}
Instead of requiring $s_i\le s_{i+1}$ in condition (7a), we can also require $s_i<s_{i+1}$, which yields the notion of \textit{strictly monotone double Hurwitz numbers}.
\end{Remark}

In \cite{zbMATH06586291}, a combinatorial interpolation between simple and monotone double Hurwitz numbers was introduced. The idea is to impose the monotonicity condition $(7a)$ only on the first $k$ transpositions.
\begin{Definition}
Let $d,g,\mu,\nu$ be as in \cref{def:hurwitz} and let $k$ be a non-negative integer. We define a mixed factorisation of type $(g,\mu,\nu,k)$ to be a factorisation $(\sigma_1,\tau_1,\dots,\tau_b,\sigma_2)$ of type $(g,\mu,\nu)$ satisfying the following additional condition:
\begin{enumerate}
\item [(7b)] If $\tau_i=(r_i\ s_i)$ with $r_i<s_i$, we have $s_i\le s_{i+1}$ for all $i=1,\dots,k-1$. 
\end{enumerate}
Let $\mathcal{F}(g,\mu,\nu,k)$ be the set of all mixed factorisations of type $(g,\mu,\nu,k)$. Then we define the \textit{mixed Hurwitz number} to be:
\begin{equation}
h^{(2),\le}_{k,b-k;\mu,\nu}=\frac{1}{d!}\left|\mathcal{F}(g,\mu,\nu,k)\right|.
\end{equation}
\end{Definition}

Fixing the length $\mu$ and $\nu$, we can view mixed Hurwitz numbers as a function
\begin{align}
h^{(2),\le}_{k,b-k}:\mathbb{N}^{\ell(\mu)}\times\mathbb{N}^{\ell(\nu)}&\to\mathbb{Q}\\(\mu,\nu)&\mapsto h^{(2),\le}_{k,b-k;\mu,\nu},
\end{align}
where $\ell(\mu)$ (resp. $\ell(\nu)$) is the length of $\mu$ (resp. $\nu$). For each $I\subset\{1,\dots,\ell(\mu)\},J\subset\{1,\dots,\ell(\nu)\}$ we obtain linear equations $\sum_{i\in I}\mu_i-\sum_{j\in J}\nu_j=0$, where the $\mu_i$ (resp. $\nu_j$) are the coordinates in $\mathbb{N}^{\ell(\mu)}$ (resp. $\mathbb{N}^{\ell(\nu)}$). The equations induce a hyperplane arrangement $\mathcal{W}$ in $\mathbb{N}^{\ell(\mu)}\times\mathbb{N}^{\ell(\nu)}$. By considering the complement on $\mathcal{W}$ this hyperplane arrangement divides $\mathbb{N}^{\ell(\mu)}\times\mathbb{N}^{\ell(\nu)}$ into chambers $C$.

\begin{thm}[\cite{GJV}, \cite{zbMATH06586291}]
\label{thm:gjv}
The function $h^{(2),\le}_{k,b-k}$ described above is piecewise polynomial, i.e. for each chamber $C$ there exists a polynomial $P^{(2),\le}_{k,b-k}(C)\in\mathbb{Q}[\underline{M},\underline{N}]$, where $\underline{M}=M_1,\dots,M_{\ell(\mu)}$ and $\underline{N}=N_1,\dots,N_{\ell(\nu)}$, such that $h^{(2),\le}_{k,b-k;\mu,\nu}=P^{(2),\le}_{k,b-k}(C)(\mu,\nu)$.
\end{thm} 

\subsection{Hurwitz numbers in terms monodromy graphs}
\label{subsec:trophur}

In this section, we recall the connection between Hurwitz numbers and so-called monodromy graphs. As a first step we associate maps between graphs to certain factorisations in the symmetric group in the following construction.

\begin{Construction}
\label{ref:consim}
Let $(\sigma_1,\tau_1,\dots,\tau_b,\sigma_2)$ be a factorisation of type $(g,\mu,\nu)$ with $\sigma_1$ as in \cref{equ:per}. We associate a graph with labeled vertices and edges and with a map to the interval $[0,\dots,b+1]$ as follows:\vspace{\baselineskip}\\
\textit{Constructing the graph.}
\begin{enumerate}
\item We start with $\ell(\mu)$ vertices over $0$, labeled by $\sigma_1^1,\dots,\sigma_1^{\ell(\mu)}$. We will call these vertices \textit{in-ends}. Moreover, we attach an edge $e_v$ to each vertex $v$ over $0$ which maps to $(0,1)$. We label these edge attached to the vertex labeled $\sigma_1^j$ by the same label. 
\item We define $\Sigma_{i+1}=\tau_i\cdots\tau_1\sigma_1$ for $i=1,\dots,m$, $\Sigma_0=\sigma_1$ and $\Sigma_{b+1}=(\sigma_2)^{-1}$. Comparing $\Sigma_i$ and $\Sigma_{i+1}$, the transposition $\tau_{i}$ either joins two cycles of $\Sigma_i$ or cuts one cycle in two.
\end{enumerate}

Assuming the preimage of $[0,i)$ has been constructed, repeat the steps (3)--(4) until $i=b$:
\begin{enumerate}
\item [(3a)] [\textbf{Join}] If $\tau_i$ joins the cycles $\Sigma_{i-1}^s$ and $\Sigma_{i-1}^{s'}$ to a new cycle $\Sigma'$, we create a vertex over $i$ labeled $\tau_i$. This vertex is joined with the edges corresponding to $\Sigma_{i-1}^s$ and $\Sigma_{i-1}^{s'}$. These edges map to some interval $(a_s,i)$ and $(a_{s'},i)$ respectively, where $a_{s}$ (resp. $a_{s'}$) is the image of the other vertex adjacent to the edge corresponding to $\Sigma_{i-1}^{s}$ (resp. $\Sigma_{i-1}^{s'}$). We call those edges the \textit{incoming edges at }$\tau_i$. Moreover, we attach an edge to $\tau_i$ mapping to $(i,i+1)$, which we label by $\Sigma'$. We call this edge the \textit{outgoing edge at }$\tau_i$.
\item [(3b)] [\textbf{Cut}] If $\tau_i$ cuts $\Sigma_{i-1}^s$ into $\Sigma'$ and $\Sigma''$, we create a vertex over $i$ labeled $\tau_i$. We attach one edge connecting $\tau_i$ to the edge corresponding to $\Sigma_{i-1}^s$, which maps to $(a_s,i)$ as above and attach two edges mapping to $(i,i+1)$ labeled $\Sigma'$ and $\Sigma''$ respectively. As above, we call the edge mapping to $(a_s,i)$ the \textit{ingoing edge at }$\tau_i$ and the edges mapping to $(i,i+1)$ \textit{outgoing edges at} $\tau_i$.
\item [(4)] We extend those edges which so far are only adjacent to one vertex, such that the edge $e$ maps to $(a_e,i+1)$, where $a_e$ is the image of the vertex adjacent to $e$.
\item [(5)] When $i=b$ is reached, the leaves of the graph which are not adjacent to in-ends correspond to the cycles of $\Sigma_{b+1}$. We create vertices over $b+1$ which we label $(\sigma_2^{-1})^1,\dots,(\sigma_2^{-1})^{\ell(\nu)}$ and connect the corresponding edges to those vertices.
\end{enumerate}
\textit{Relabelling the graph.}
\begin{enumerate}
\item [(6)] We drop the labels $\tau_i$ at the vertices of $1,\dots,b$.
\item [(7)] We label the in-ends (resp. out-ends) by $1,\dots,\ell(\mu)$ (resp. $1,\dots,\ell(\nu)$) according to the labels of $\sigma_1$ and $\sigma_2$.
\item [(8)] If a vertex or an edge is labeled by a cycle $\sigma$, we replace the label by the length of the cycle.
\end{enumerate}

We obtain a graph $\Gamma$ with a map to $[0,b+1]$. We call $\Gamma$ together with the map the \textit{monodromy graph of type} $(g,\mu,\nu)$ \textit{associated to} $(\sigma_1,\tau_1,\dots,\tau_b,\sigma_2)$.
\end{Construction}

An example for a more general version of this construction can be found in \cref{ex:graph}.

\begin{Remark}
We can view the graph $\Gamma$ as an abstract tropical curve, where we can choose arbitrary lengths. Furthermore, the line associated to the interval $[0,b+1]$ is an abstract tropical curve as well (it is called the \textit{tropical projective line}: the point $0$ in the interval corresponds to $0$ in the projective line, the point $[b+1]$ in the interval corresponds to $\infty$ in the projective line). The map $\Gamma\to[0,b+1]$ can be viewed as a tropical cover of degree $d$, where the edges adjacent to vertices over $0$ (resp. $\infty$) yield the profile $\mu$ (resp. $\nu$). The ramification points in $\Gamma$ are the $3-$valent vertices and the branch points in $[0,b+1]$ are their images.
\end{Remark}

With the next definition we obtain a classification of the graphs we obtain from \cref{ref:consim}.
\begin{Definition}
A monodromy graph $\varGamma$ of type $(g,\mu,\nu)$ is a graph with a map to $[0,b+1]$ (where $b=2g-2+\ell(\mu)+\ell(\nu)$) with the following properties:\vspace{\baselineskip}\\
   \textit{Graph/Map conditions.}
 \begin{enumerate}
  \item The graph $\Gamma$ is a connected.
  \item The first Betti number of $\Gamma$ is $g$.
  \item The map sends vertices to integers, we call the image $i$ of a vertex its position. Moreover, the map sends edges to open intervals. For a vertex of position $i$, we call edges mapped to $(a,i)$ for $a<i$ \textit{incoming edges at }$i$ and edges mapped to $(i,a)$ for $a>i$ \textit{outgoing edges at }$i$.
  \item The graph has $\ell(\mu)+\ell(\nu)$ leaves. There are $\ell(\mu)$ leaves mapped to $0$ labeled by $1,\dots,\ell(\mu)$ and $\ell(\nu)$ leaves over $b+1$ labeled by $1,\dots,\ell(\nu)$.
  \item Over each integer $i\in[0,b]$, there is exactly one vertex which locally looks like one of the graphs in \cref{fig:ver}. We call these vertices \textit{inner vertices}.
 \end{enumerate}
\begin{figure}[ht]
\begin{center}
\begin{tikzpicture}
\draw (-2,-2)--(0,0); 
\draw (-2,2)--(0,0);
\draw (0,0)--(2,0);
\draw (-2,-3)--(2,-3);
\draw[fill=black] (0,-3) circle (2pt) node[anchor=south east] {$i$};
\end{tikzpicture}\hspace{15pt}
\begin{tikzpicture}
\draw (-2,0)--(0,0); 
\draw (2,2)--(0,0);
\draw (0,0)--(2,-2);
\draw (-2,-3)--(2,-3);
\draw[fill=black] (0,-3) circle (2pt) node[anchor=south east] {$i$};
\end{tikzpicture}
\caption{Local structure of the map for a monodromy graph.}
\label{fig:ver}
\end{center}
\end{figure}
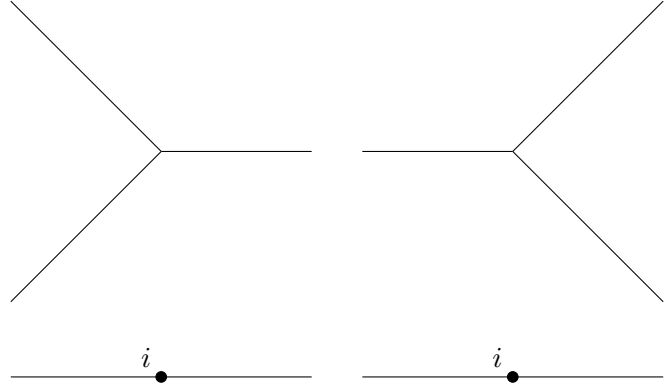
 \textit{Weight conditions.}
 \begin{enumerate}
  \setcounter{enumi}{5}
  \item We assign a positive integer weight $\omega(e)$ to each edge $e$. The in-end labeled $i$ has weight $\mu_i$. The out-end labeled $j$ has weight $\nu_j$.
  \item At each inner vertex, the sum of the weights of incoming edges equals the sum of the weights of outgoing edges.
  \end{enumerate}
This is the \textit{balancing condition for monodromy graphs}, which comes from the observation that by definition a monodromy graph is a combinatorial type of a tropical morphism (see e.g. \cite{BBMopenhurwitz}).
An isomorphism of monodromy graphs $\Gamma\to[0,b+1]$ and $\Gamma'\to[0,b+1]$ of type $(g,\mu,\nu)$ is a graph isomorphism $f:\Gamma\to\Gamma'$, such that
\begin{equation}
\begin{tikzcd}
\Gamma \arrow{rr}{f} \arrow {dr} & & \arrow{dl} \Gamma'\\
& {[0,b+1]} &
\end{tikzcd}
\end{equation}
commutes.
\end{Definition}

We now define tropical simple double Hurwitz numbers in terms of monodromy graphs.

\begin{Definition}
Let $d$ be a positive integer, $\mu,\nu$ two ordered partitions of $d$ and let $g$ be a non-negative integer. Furthermore, let $b=2g-2+\ell(\mu)+\ell(\nu)$, then we define the \textit{tropical simple double Hurwitz numbers}
\begin{equation}
\mathbb{T}h^{(2)}_{b;\mu,\nu}=\sum\frac{1}{\left|\mathrm{Aut}\left(\Gamma\to[0,b+1]\right)\right|}\prod\omega(e),
\end{equation}
where we sum over all monodromy graphs $\Gamma\to[0,b+1]$ of type $(g,\mu,\nu)$ and take the product of all inner edges $e$ of $\Gamma$.
\end{Definition}

Their relation to the algebro-geometric counterparts is given by the following theorem. It is proved by analysing how many factorisations of type $(g,\mu,\nu)$ yield the same monodromy graph of type $(g,\mu,\nu)$.

\begin{thm}[\cite{CJM}]
Let $d$ be a positive integers, $\mu,\nu$ two ordered partitions of $d$ and let $g$ be a non-negative integer. Furthermore, let $b=2g-2+\ell(\mu)+\ell(\nu)$, then
\begin{equation}
h^{(2)}_{b;\mu,\nu}=\mathbb{T}h^{(2)}_{b;\mu,\nu}
\end{equation}
\end{thm}

A similar approach was taken in \cite{do2015monotone} in order to give a tropical interpretation for monotone double Hurwitz numbers, when $\nu=(a,\dots,a)$ for some positive integer $a$. We generalise this approach to monotone double Hurwitz numbers (and more general cases) in \cref{sec:gr}.

\section{Triply interpolated Hurwitz numbers via monodromy graphs}
\label{sec:gr}
We define triply interpolated Hurwitz numbers, which are a combinatorial interpolation between simple, monotone and strictly monotone Hurwitz numbers.

\begin{Definition}
Let $g$ be a non-negative integer, $\mu,\nu$ partitions of the same positive integer, $p,q,r$ non-negative integers, such that $p+q+r=2g-2+\ell(\mu)+\ell(\nu)$. A triply interpolated factorisation of type $(g,\mu,\nu,p,q,r)$ is a factorisation of type $(g,\mu,\nu)$ (see conditions (1)--(6) \cref{def:hursym}), such that:
\begin{enumerate}
\item[(7c)] If $\tau_i=(r_i\;s_i)$ with $r_i<s_i$, we have $s_{i+1}\ge s_{i}$ for $i=1,\dots,p-1$ and $s_{i+1}>s_{i}$ for $i=p,\dots,p+q-1$.
\end{enumerate}
We denote the set of all triply interpolated factorisations by $\mathcal{F}(g,\mu,\nu,p,q,r)$ and define the triply interpolated Hurwitz number by
\begin{equation}
H_{p,q,r;\mu,\nu}^{\le,<,(2)}=\frac{1}{d!}\left|\mathcal{F}(g,\mu,\nu,p,q,r)\right|.
\end{equation}
\end{Definition}

\begin{Remark}
We note that for $p=q=0$, we obtain the simple double Hurwitz numbers, for $q=r=0$ the monotone double Hurwitz numbers and for $p=r=0$ the strictly monotone double Hurwitz numbers associated to the data.
\end{Remark}

In \cite{do2015monotone}, Lemma $7$, the following result was proved for monotone factorisations. It may be easily adapted for triply interpolated factorisations. For the convenience of the reader we give the proof. We work in the symmetric group ring. For an introduction to these techniques and their connection to Hurwitz numbers, see e.g. \cite{CM} Chapter 9.

\begin{Lemma}
\label{lem:indep}
Fix a permutation $\sigma$ of cycle type $\mu$ and non-negative integers $p,q,r$. The number of triply interpolated factorisation $(\sigma_1,\tau_1,\dots,\tau_r,\sigma_2)$ of type $(g,\mu,\nu,p,q,r)$, satisfying $\sigma_1=\sigma$ does not depend on the choice of $\sigma$ for $p=0$ or $q=0$.
 \end{Lemma}

\begin{proof}
Fix a permutation $\sigma$ of cycle type $\mu$. Let $K_{\tau}^\bullet(\sigma)$ be the number of triply interpolated factorisations $(\sigma,\tau_1,\dots,\tau_b,\sigma_2)$ of type $\tau=(g,\mu,\nu,p,q,r)$ where we drop the transitivity condition. We can rewrite the equation as follows 
\begin{equation}
\sigma_2^{-1}\tau_b\cdots\tau_1=\sigma^{-1}.
\end{equation}
We see that $K_{\tau,k}^\bullet(\sigma)$ is the coefficient of $\sigma^{-1}$ in
\begin{align}
\label{equ:conj}
C_{\nu}h_p(J_2,\dots,J_{|\nu|})(C_{\kappa})^{r}\in\mathbb{C}[\mathcal{S}_d]\textrm{ for }q=0,\\
\label{equ:conj2}
C_{\nu}\sigma_q(J_2,\dots,J_{|\nu|})(C_{\kappa})^{r}\in\mathbb{C}[\mathcal{S}_d]\textrm{ for }p=0,
\end{align}
where $\kappa=(2,1,\dots,1)\vdash|\nu|$, $C_{w}$ denotes the conjugacy class of permutations with cycle type $w$, $h_i$ is the complete homogeneous symmetric polynomial of degree $i$, $\sigma_i$ is the elementary homogeneous symmetric polynomial of degree $i$ and $J_i$ denote the Jucys-Murphy elements \begin{equation}
J_i=(1,i)+\dots+(i-1,i)\in\mathbb{C}[\mathcal{S}_d]
\end{equation} for $i=2,\dots,|\nu|$. It is well known, that conjugacy classes and the symmetric polynomials in the Jucys-Murphy elements lie in the center of $\mathbb{C}[\mathcal{S}_d]$. Thus the expressions in \cref{equ:conj} and \cref{equ:conj2} are a linear combination of conjugacy classes and therefore all permutations in the same conjugacy class appear with the same coefficient. Thus $K_{\tau}^\bullet(\sigma)$ only depends on the conjugacy class of $\sigma$.\
Now let $K_{\tau,k}^\circ(\sigma)$ be the number of factorisations as above that satisfy the transitivity condition. If $\sigma$ is a $d-$cycle, where $d=\sum\mu_i$, then $K_{\tau}^\bullet(\sigma)=K_{\tau}^\circ(\sigma)$ and the result holds. For any permutation $\sigma$, set $\sigma=\Sigma_1\cdots\Sigma_{\ell(\mu)}$ be the decomposition in disjoint cycles. We can decompose every non-transitive factorisation into a union of transitive factorisations. This leads to the following formula
  \begin{equation}K_{\tau}^\bullet(\sigma)=K_{\tau}^\circ(\sigma)+\sum_{s=2}^{\ell(\mu)}\sum_{\substack{I_1\sqcup\dots\sqcup I_s=[\ell(\mu)]\\\mu^{(1)}\sqcup\dots\sqcup\mu^{(s)}=\mu\\\nu^{(1)}\sqcup\dots\sqcup\nu^{(s)}=\nu}}\sum_{\substack{p_1+\dots+p_s=p\\q_1+\dots+q_s=q\\r_1+\dots+r_s=r\\g_1+\dots+g_s=g}}\prod_{l=1}^sK_{\tau_l}^\circ(\Sigma_{I_l}),\end{equation}
where the summation is over partitions of $[\ell(\mu)]=\{1,\dots,\ell(\mu)\}$ into disjoint non-empty subsets $I_1\sqcup\dots\sqcup I_s$, ordered tuples of partitions $\mu^{(1)}\sqcup\dots\sqcup\mu^{(s)}$ (resp. $\nu^{(1)}\sqcup\dots\sqcup\nu^{(s)}$) whose union is $\mu$ (resp. $\nu$), such that $|\nu^l|=|\mu^l|$ and 
\begin{equation}
g_l=\frac{p_l+q_l+r_l+2-\ell(\mu)-\ell(\nu)}{2}
\end{equation}
and where we use the Notation $\tau_l=(g_l,\mu^{(l)},\nu^{(l)},p_l,q_l,r_l)$. Moreover, for $I\subset[\ell(\mu)]$, we let $\Sigma_I$ denote the permutation obtained by taking the product of all cycles $\Sigma_i$ for $i\in I$.
 We already proved that $K_{\tau,k}^\bullet(\sigma)$ only depends on the cycle type of $\sigma$ for $p=0$ or $q=0$, by induction on the length of $\mu$ the terms $K_{\tau}^\circ(\Sigma_{I_l})$ only depend on the cycle type of $\Sigma_{I_l}$, thus $K_{\tau,k}^\circ(\sigma)$ only depends on the cycle type of $\sigma$ and we are finished.
\end{proof}

Thus, in order to compute $H_{p,q,r;\mu,\nu}^{\le,<,(2)}$ --- when $p=0$ or $q=0$ --- we do not have to count all triply interpolated factorisations of type $(g,\mu,\nu,p,q,r)$, rather we may compute the number of triply interpolated factorisations $(\sigma_1,\tau_1,\dots,\tau_b,\sigma_2)$ of type $(g,\mu,\nu,p,q,r)$ with fixed $\sigma_1$ and multiply this number by $\frac{1}{d!}\cdot|\{\sigma\in S_d:\mathcal{C}(\sigma)=\mu\}|$ to obtain $H_{p,q,r;\mu,\nu}^{\le,<,(2)}$. We can thus simplify this counting problem with a smart choice of $\sigma_1$ (see (cref{equ:per}), We translate the counting problem to a problem of counting monodromy graphs as in \cite{CJM}, \cite{CJMa} and \cite{do2015monotone}. In the latter, the choice of $\sigma_1$ as in \cref{equ:per} was already utilised. To give our description of triply interpolated Hurwitz numbers in terms of monodromy graphs, we make the following choice for fixed $\mu=(\mu_1,\dots,\mu_{\ell(\mu)})$:
\begin{equation}
\label{equ:per}
\sigma=(1\cdots\mu_1)(\mu_1+1\cdots\mu_1+\mu_2)\cdots\left(\sum_{i=1}^{\ell(\mu)-1}\mu_i+1\cdots\sum_{i=1}^{\ell(\mu)}\mu_i\right),
\end{equation}
where the cycle $\sigma_1^s=\left(\sum_{i=1}^{s-1}\mu_i+1\cdots\sum_{i=1}^{s}\mu_i\right)$ is labeled by $s$. We define $M_{p,q,r;\mu,\nu}^{\le,<,(2)}$ to be the number of triply interpolated factorisations of type $(g,\mu,\nu,p,q,r)$ with $\sigma_1$ as in \cref{equ:per}. The number of permutations of cycle type $\mu$ with labeled cycles is 
\begin{equation}
\epsilon(\mu)=\frac{d!}{\mu_1\cdots \mu_{\ell(\mu)}}\end{equation} and we see that 
\begin{equation}\label{equ:handm}
H_{p,q,r;\mu,\nu}^{\le,<,(2)}=\frac{1}{d!}\epsilon(\mu)M_{p,q,r;\mu,\nu}^{\le,<,(2)}=\frac{1}{\mu_1\cdots \mu_{\ell(\mu)}}M_{p,q,r;\mu,\nu}^{\le,<,(2)}
\end{equation}
for $p=0$ or $q=0$. In particular \cref{equ:handm} is true for the extremal cases of monotone and strictly monotone Hurwitz numbers.\vspace{\baselineskip}

We will express $M_{p,q,r;\mu,\nu}^{\le,<,(2)}$ in terms of monodromy graphs and begin by associating a graph to a triply interpolated factorisation.
\begin{Construction}
\label{ref:con}
Let $(\sigma_1,\tau_1,\dots,\tau_b,\sigma_2)$ be a triply interpolated factorisation of type $(g,\mu,\nu,\allowbreak p,q,r)$ with $\sigma_1$ as in \cref{equ:per}. We equip the graph we obtain from \cref{ref:consim} steps (1)--(5) with additional structure:\vspace{\baselineskip}

\textit{Colouring the graph.}
\begin{enumerate}
\item [(6)]We colour all edges \textit{normal}.
\item [(7)]We colour all edges adjacent to in-ends \textit{dashed}.
\item [(8)]We repeat steps [(9a)] and [(9b)] for all transpositions $\tau_1,\dots,\tau_{p+q}$.
\item [(9a)] [\textbf{Cut}] If $\tau_i$ is a transposition as in (3a), then we assume $\tau_i=(a\ b)$ with $a<b$ and $a\in\Sigma_{i-1}^s$ and $b\in\Sigma_{i-1}^{s'}$. Then we colour the edge labeled $\Sigma_{i-1}^{s'}$ \textit{bold} and the outgoing egde at $\tau_i$ \textit{dashed}.
\item [(9b)] [\textbf{Join}] If $\tau_i$ is a transposition as in (3b), we colour the edge labeled $\Sigma_{i-1}^s$ \textit{bold}. For $\tau_i=(a\ b)$ with $a<b$, we colour the outgoing edge corresponding to the cycle containing $b$ \textit{dashed}.
\end{enumerate}
\textit{Distributing counters.}\\
We distribute a counter to all non-normal edges.
\begin{enumerate}
\item [(10)] We start by distributing $1$ to all edges adjacent to in-ends.
\item [(11)] For $\tau_i=(r_i\ s_i)$, where $r_i<s_i$, there is a unique way of expressing $s_i$ as follows: \begin{equation}s_i=\sum_{j=1}^l\mu_j+c,\end{equation} where $c<\mu_{j+1}$. Then we distribute $c$ to the outgoing dashed/bold edge adjacent to the vertex labeled $\tau_i$.
\end{enumerate}
\textit{Relabelling the graph.}
\begin{enumerate}
\item [(12)] We drop the labels $\tau_i$ at the vertices of $1,\dots,m$.
\item [(13)] We label the in-ends (resp. out-ends) by $1,\dots,\ell(\mu)$ (resp. $(1,\dots,\ell(\nu)$) according to the labels of $\sigma_1$ and $\sigma_2$.
\item [(14)] If a vertex or an edge is labeled by a cycle $\sigma$, we replace the label by the length of the cycle.
\end{enumerate}
We obtain a graph $\Gamma$. We call $\Gamma$ the \textit{triply interpolated monodromy graph of type} $(g,\mu,\nu,p,q,r)$ \textit{associated to} $(\sigma_1,\tau_1,\dots,\tau_b,\sigma_2)$.
\end{Construction}

\begin{Example}
\label{ex:graph}
On the left of \cref{fig:tu}, we illustrate the cut-and-join process for the following factorisation of type $(1,(2,2),(4))$: \begin{equation}((12)(34),(12),(23),(13),(1243)).\end{equation} In fact, this is a monotone factorisation, which can be viewed as a triply interpolated factorisation of type $(1,(2,2),(4),3,0,0)$ and the associated triply interpolated monodromy graph is illustrated on the right.
\end{Example}

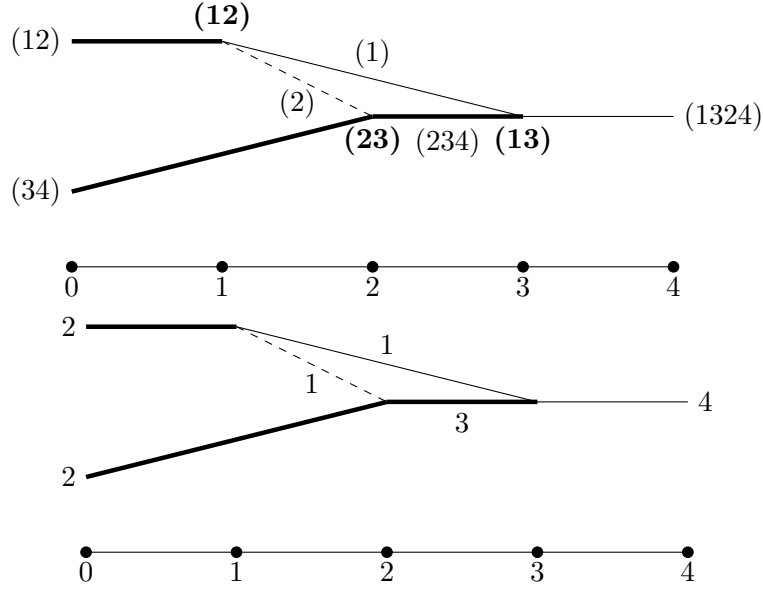
\begin{figure}
\begin{center}
\begin{tikzpicture}
\draw [line width=0.6mm] (0,0)--(2,0);
\draw [dashed] (2,0)--(4,-1) node[midway,anchor=north] {$(2)$};
\draw (2,0)--(6,-1) node[midway,anchor=south] {$(1)$};
\draw [line width=0.6mm] (4,-1)--(6,-1) node[midway,anchor=north] {$(234)$};
\draw (6,-1)--(8,-1);
\draw [line width=0.6mm] (0,-2)--(4,-1);
\draw (0,-3)--(8,-3);
\draw (0,0) node[anchor=east] {$(12)$};
\draw (2,0) node[anchor=south] {$\textbf{(12)}$};
\draw (4,-1) node[anchor=north] {$\textbf{(23)}$};
\draw (6,-1) node[anchor=north] {$\textbf{(13)}$};
\draw (0,-2) node[anchor=east] {$(34)$};
\draw (8,-1) node[anchor=west] {$(1324)$};
\draw[fill] (0,-3) circle (2pt) node[anchor=north] {$0$};
\draw[fill] (2,-3) circle (2pt) node[anchor=north] {$1$};
\draw[fill] (4,-3) circle (2pt) node[anchor=north] {$2$};
\draw[fill] (6,-3) circle (2pt) node[anchor=north] {$3$};
\draw[fill] (8,-3) circle (2pt) node[anchor=north] {$4$};
\end{tikzpicture}
\begin{tikzpicture}
\draw [line width=0.6mm] (0,0)--(2,0);
\draw [dashed] (2,0)--(4,-1) node[midway,anchor=north] {$1$};
\draw (2,0)--(6,-1) node[midway,anchor=south] {$1$};
\draw [line width=0.6mm] (4,-1)--(6,-1) node[midway,anchor=north] {$3$};
\draw (6,-1)--(8,-1);
\draw [line width=0.6mm] (0,-2)--(4,-1);
\draw (0,-3)--(8,-3);
\draw (0,0) node[anchor=east] {$2$};
\draw (0,-2) node[anchor=east] {$2$};
\draw (8,-1) node[anchor=west] {$4$};
\draw[fill] (0,-3) circle (2pt) node[anchor=north] {$0$};
\draw[fill] (2,-3) circle (2pt) node[anchor=north] {$1$};
\draw[fill] (4,-3) circle (2pt) node[anchor=north] {$2$};
\draw[fill] (6,-3) circle (2pt) node[anchor=north] {$3$};
\draw[fill] (8,-3) circle (2pt) node[anchor=north] {$4$};
\end{tikzpicture}
\caption{In the upper graph, the bold permutations correspond to transposition $\tau_i$, the other ones correspond to the cycles of the permutations $\tau_i\cdots\tau_1\sigma_1$. In the lower graph, the non-normal edges are bi-labeled, where the first number is the weight and the second is the counter.}
\label{fig:tu}
\end{center}
\end{figure}

We now classify the graphs we obtain from \cref{ref:con}. Moreover, we will understand how many triply interpolated factorisations yield the same monodromy graph. This result will be our main tool in the discussion of polynomiality.

\begin{Definition}
\label{def:mono}
A triply interpolated monodromy graph $\Gamma$ of type $(g,\mu,\nu,p,q,r)$ is a monodromy graph of type $(g,\mu,\nu)$ (where $p+q+r=2g-2+\ell(\mu)+\ell(\nu)$) with the following properties:\vspace{\baselineskip}

\textit{Colouring conditions.}\\
    The following conditions are only applied to edges adjacent to the first $p+q$ inner vertices.
  \begin{enumerate}
   \setcounter{enumi}{5}
  \item We colour the edges of the graph by the three colours: normal, bold and dashed, such that each inner vertex is one of the six types in \cref{fig:gra}.
  \item There are no normal in-ends.
  \item We call a connected path of bold edges beginning at an in-end a \textit{chain}.
  \item Let $C$ and $C'$ be two chains and let $f_C$ (resp. $f_{C'}$) be the position of the first inner vertex of $C$ (resp. $C'$) and let $l_C$ (resp. $l_{C'}$) be the position of the last inner vertex of $C$ (resp. $C'$). Then we require the intervals $[f_C,l_C]$ and $[f_{C'},l_{C'}]$ to have empty intersection.
  \item The intervals $[f_C,l_C]$ induce a natural ordering on the chains, namely $C<C'$ if $f_C<f_{C'}$. We require this ordering to be compatible with the ordering of the partition $\mu$ as follows: Let $C_1$ and $C_2$ be two chains of bold edges, $i_1$ and $i_2$ the respective in-ends, then we demand $C_1<C_2$ if and only if $i_1<i_2$.
  \end{enumerate}
  The ordering of the chains corresponds to the monotonicity condition as we will see later.\vspace{\baselineskip}\\
    \textit{Counter conditions.}
  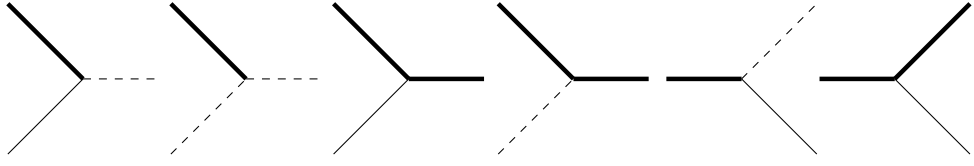
\begin{figure}[ht]
  \begin{tikzpicture}
  \draw [line width=0.6mm] (0,1)--(1,0);
  \draw (0,-1)--(1,0);
  \draw[dashed] (1,0)--(2,0);
  \end{tikzpicture} \begin{tikzpicture}
  \draw [line width=0.6mm] (0,1)--(1,0);
  \draw[dashed] (0,-1)--(1,0);
  \draw[dashed] (1,0)--(2,0);
  \end{tikzpicture} \begin{tikzpicture}
  \draw [line width=0.6mm] (0,1)--(1,0);
  \draw (0,-1)--(1,0);
  \draw [line width=0.6mm] (1,0)--(2,0);
  \end{tikzpicture} \begin{tikzpicture}
  \draw [line width=0.6mm] (0,1)--(1,0);
  \draw [dashed] (0,-1)--(1,0);
  \draw [line width=0.6mm] (1,0)--(2,0);
  \end{tikzpicture}\hspace{5pt}\begin{tikzpicture}
  \draw [line width=0.6mm] (0,0)--(1,0);
  \draw [dashed] (1,0)--(2,1);
  \draw (1,0)--(2,-1);
  \end{tikzpicture}\begin{tikzpicture}
  \draw [line width=0.6mm] (0,0)--(1,0);
  \draw [line width=0.6mm] (1,0)--(2,1);
  \draw (1,0)--(2,-1);
  \end{tikzpicture}
  \caption{Local colouring of the graph.}
\label{fig:gra}
\end{figure}
  \begin{enumerate}
  \setcounter{enumi}{10}
  \item We distribute a counter to each non-normal edge (thus, those inner edges are bi-labeled by the weight and the counter and the non-normal leaves are tri-labeled where the additional label is a number in $\{1,\dots,\ell(\mu)\}$ or $\{1,\dots,\ell(\nu)\}$).
 \item The counter for each in-end is set to $1$. 
 \item At each inner vertex $v$ mapping to $k$, there is a unique incoming bold edge, whose counter we denote by $i_k$ and a unique out-going non-normal edge, whose counter we denote by $o_k$. For $k=1,\dots,p$ we require $i_k\le o_k$ and for $k=p+1,\dots,p+q$, we require $i_k<o_k$.
 \item  Every non-normal edge arises from a unique chain of bold edges: Every bold edge is part of a unique chain and every dashed edge is sourced at a unique chain. Let the non-normal edge $e$ arise from the chain starting at the in-end labeled $i$. The counter $l_e$ of the non-normal edge $e$ is smaller or equal than $\mu_i$ and greater than $\mu_i-\omega(e)$.
 \end{enumerate}
 The last condition reflects that these cycles corresponding to each such edge $e$ should contain at least $\mu_i-(l_e-1)$ elements.
\end{Definition}

\begin{Definition}
Let $\Gamma$ be a triply interpolated monodromy graph of type $(g,\mu,\nu,p,q,r)$. We call the graph we obtain by removing the counters the \textit{reduced monodromy graph of }$\Gamma$.\\
A graph $\Gamma$ that appears as a triply interpolated monodromy graph of type $(g,\mu,\nu,p,q,r)$ without counters is called a \textit{reduced monodromy graph of type }$(g,\mu,\nu,p,q,r)$.
\end{Definition}

We called the graph we obtained from \cref{ref:con} a triply interpolated monodromy graph. The following Lemma justifies the choice of this term.

\begin{Lemma}
\label{le:fir}
The graphs obtained from \cref{ref:con} are triply interpolated monodromy graphs in the sense of \cref{def:mono}.
\end{Lemma}

\begin{proof}
Let $(\sigma_1,\tau_1,\dots,\tau_b,\sigma_2)$ be a triply interpolated factorisation of type $(g,\mu,\nu,p,q,r)$ with $\sigma_1$ as in \cref{equ:per}. The conditions (1)-(7) in \cref{def:mono} follow immediately by Construction.\\
The chains of bold edges correspond to the following situation in the symmetric group setting: Suppose $\tau_i=(r_i\ s_i)$ for $r_i<s_i$. Since we chose $\sigma_1$ in \cref{equ:per}, we can group the transpositions $\tau_i$ for $i\le k$. We say $\tau_i$ is of type $t$, if $s_i$ is contained in the cycle of $\sigma_1$ labeled $t$. Now, for $i<j$, let $t_i$ (resp. $t_j$) be the type of $\tau_i$ (resp. $\tau_j$), then $t_i\le t_j$.\\
A chain of bold edges starting at the in-end $i$ corresponds to the transpositions of type $i$. Thus conditions (8)-(10) follow.\\
The counter conditions (11)-(13) follow by Construction. For condition (14) we observe the following: Let $e$ be a non-normal edge which arises from the chain of bold edges $C$ and whose source vertex has position $p$. Let $C$ start at $i$ and let $l_e$ be the counter of $e$. Moreover, let the $i-$th cycle of $\sigma_1$ be of the form \begin{equation}\left(\sum_{a=1}^{i-1}\mu_a+1\cdots\sum_{a=1}^{i}\mu_a\right).\end{equation}
Then for $\tau_p=(r_p\ s_p)$ we have $s_p=\sum_{a=1}^{i-1}\mu_a+l_e$. By monotonicity multiplying $\sigma_1$ by $\tau_1\cdots\tau_p$ does not change the images of $\sum_{a=1}^{i-1}\mu_a+l_e,\dots,\sum_{a=1}^{i}\mu_a-1$. Thus the cycle of $\tau_p\cdots\tau_1\sigma_1$ containing $\sum_{a=1}^{i-1}\mu_a+l_e$ has the following structure \begin{equation}\left(\cdots\sum_{a=1}^{i-1}\mu_a+l_e\cdots\sum_{a=1}^{i}\mu_a\cdots\right),\end{equation} where the dots left and right indicate other elements.
Thus the weight $\omega(e)$ of the weight $e$ fulfils the following inequality\begin{equation}\omega(e)\ge\mu_i-l_e+1\end{equation} or equivalently \begin{equation}l_e>\mu_i-\omega(e).\end{equation} Thus condition (14) is fulfilled as well.
\end{proof}

\begin{Definition}
An automorphism of a triply interpolated monodromy graph $\Gamma$ is a graph automorphism $f:\Gamma\to\Gamma$, such that:
\begin{enumerate}
\item The function $f$ respects labels, weights, colours and counters.
\item The following diagram commutes:
\begin{center}
$\begin{CD}
\Gamma     @>{f}>>  \Gamma\\
@VVV        @VVV\\
[0,m+1]    @>{\mathrm{id}}>>  [0,m+1].
\end{CD}$
\end{center}
\end{enumerate}
\vspace{\baselineskip}
We denote the automorphism group of $\Gamma$ by $\mathrm{Aut}(\Gamma)$.\end{Definition}

We are now ready to give a weighted bijection between triply interpolated factorisations and triply interpolated monodromy graphs of type $(g,\mu,\nu,p,q,r)$.
\begin{Lemma}
\label{lem:mult}
 Let $\Gamma$ be a triply interpolated monodromy graph of type $(g,\mu,\nu,p,q,r)$. The number $m(\Gamma)$ of triply interpolated factorizations of type $(g,\mu,\nu,p,q,r)$ with $\sigma_1$ as in \cref{equ:per} for which \cref{ref:con} produces $\Gamma$ is
 \begin{equation}m(\Gamma)=\frac{1}{|\mathrm{Aut}(\Gamma)|}\prod \omega(e),\end{equation}
  where we take the product over all dashed and normal edges $e$, which are not adjacent to out-ends.\\
 We call $m(\Gamma)$ the \textit{multiplicity} of $\Gamma$.
 \end{Lemma}
\begin{Remark}
An immediate consequence of this Lemma is that the number $m(\Gamma)$ does not depend on the counters of $\Gamma$. We will use this in \cref{sec:pie}.
\end{Remark}
 \begin{proof}
 Let $v$ be one of the first $p+q$ inner vertices. If $v$ is a cut, the corresponding transposition is uniquely defined by the weights of the outgoing edges and the counter of the outgoing dashed or bold edge. If two edges are joined at $v$, the larger entry of the corresponding transposition is uniquely defined by the counter of the outgoing non-normal edge and the source chain of the in-going bold edge. However, we have a number of possibilites for the first element of the transposition, which is exactly the weight of the non-bold ingoing edge.\\
  Now let $v$ be an inner vertex whose position is greater than $p+q$. If $v$ is a cut with ingoing edge $e$, there are $\omega(e)$ possibilities for $\tau_v$, except when $\omega(e)=2n$ and both outgoing edges have weight $n$. Then, there are only $n$ possibilities for $\tau_v$. If both cycles have distinguishable evolution, it matters which cycle has which evolution and obtain a factor of $2$. If the cycles have undistinguishable evolution, this corresponds to a contribution of $\mathrm{Aut}(\Gamma)$.\\
  If $v$ is a join with ingoing edges $e$ and $e'$, the number of possibilites for $\tau_v$ is $\omega(e)\cdot\omega(e')$. Thus the Lemma is proved.
 \end{proof}

\begin{Example}
The multiplicity of the graph in the right of \cref{fig:tu} is $2$.
\end{Example}

By our previous discussion we can compute triply interpolated Hurwitz numbers in terms of triply interpolated monodromy graphs.

\begin{Proposition}
Let $p,q,r$ be positive integers and $\mu,\nu$ ordered partitions of the same number $d$. Then:
\begin{equation}
H_{p,q,r;\mu,\nu}^{\le,<,(2)}=\frac{1}{\mu_1\cdots\mu_{\ell(\mu)}}\sum_{\Gamma}m(\Gamma)\textrm{ for }p=0\textrm{ or }q=0,
\end{equation}
where we sum over all triply interpolated monodromy graphs of type $(g,\mu,\nu,p,q,r)$.
\end{Proposition}

\begin{proof}
This is an immediate consequence of \cref{le:fir}, \cref{lem:mult} and c\ref{equ:handm}.
\end{proof}

\subsection{Piecewise polynomiality}
\label{sec:pie}
We want to use \cref{lem:mult} to study the piecewise polynomiality of triply interpolated Hurwitz numbers in the flavour of the discussion of section 6 of \cite{CJM}. We begin by studying the genus $0$ case and we will use Erhart theory to generalise these results to higher genera.

\subsubsection{The genus 0 case}
\label{sec:pi0}
For the rest of this subsection, we assume $p+q+r=-2+\ell(\mu)+\ell(\nu)$. It is our aim to show that $M_{p,q,r;\mu,\nu}^{\le,<,(2)}$ is piecewise polynomial and to provide a constructive method to compute the polynomials in each chamber. By \cref{equ:handm} this also produces a method to compute the polynomials for $H_{p,q,r;\mu,\nu}^{\le,<,(2)}$ in each chamber when $p=0$ or $q=0$.

\begin{Proposition}
\label{prop:piece}
The function 
\begin{equation}
\begin{aligned}
M_{p,q,r;\ell(\mu),\ell(\nu)}^{\le,<,(2)}:\mathbb{N}^{\ell(\mu)}\times\mathbb{N}^{\ell(\nu)} & \to\mathbb{Q}\\
(\mu,\nu) & \mapsto M_{p,q,r;\mu,\nu}^{\le,<,(2)}
\end{aligned}
\end{equation}
is piecewise polynomial, i.e. for every chamber $C$ induced by the hyperplane arrangement $\mathcal{W}$, there exists a polynomial $m_{p,q,r;\ell(\mu),\ell(\nu)}^{\le,<,(2)}(C)\in\mathbb{Q}[\underline{M},\underline{N}]$, such that $M_{p,q,r;\mu,\nu}^{\le,<,(2)}=m_{p,q,r;\mu,\nu}^{\le,<,(2)}(C)(\mu,\nu)$ for all $(\mu,\nu)\in C$.
\end{Proposition}

\begin{proof}
The proof follows from \cref{lem:ch}, \cref{lem:weight} and \cref{lem:cou}.
\end{proof}

\begin{Remark}
Note, that this Proposition does not prove that the functions for triply interpolated double Hurwitz numbers, which we denote by $H_{p,q,r;\ell(\mu),\ell(\nu)}^{\le,<,(2)}$, are polynomials for $p=0$ or $q=0$, as these function differ from the functions $M_{p,q,r;\ell(\mu),\ell(\nu)}^{\le,<,(2)}$ by a factor of $\frac{1}{\prod\mu_i}$. It follows from \cref{thm:gjv}, that the polynomials $m_{p,0,r;\ell(\mu),\ell(\nu)}^{\le,<,(2)}(C)$ contain a factor of $\prod_{i=1}^{\ell(\mu)}\mu_i$ --- however this is not true for the contribution of each graph to $m_{p,0,r;\ell(\mu),\ell(\nu)}^{\le,<,(2)}(C)$ as can be seen in \cref{ex:c}. By \cite{hahn2018wall}, the same is true for the interpolation between simple and strictly monotone Hurwitz numbers, i.e. $p=0$. It would be interesting to see how the contributions for each graph add up to a polynomial which is divisible by $\prod \mu_i$. 
\end{Remark}

We start by examining the edge weights in the equation
\begin{equation}\label{equ:pol1}
M_{p,q,r;\mu,\nu}^{\le,<,(2)}=\sum_{\Gamma}m(\Gamma)=\sum_{\Gamma}\prod\omega(e),
\end{equation}
where we sum over all triply interpolated monodromy graphs of type $(0,\mu,\nu,p,q,r)$. We want to group some of the graphs appearing in the sum.

\begin{Definition}
Let $\Gamma$ be a reduced monodromy graph of type $(g,\mu,\nu,p,q,r)$. Then we define the function 
\begin{align}
F(\Gamma,\mu,\nu,p,q,r)=\left|\{\textrm{triply interpolated monodromy graphs }\right.\\
\left.\tilde{\Gamma} \textrm{ of type }(g,\mu,\nu,p,q,r)\textrm{ with reduced monodromy graph }\Gamma\}\right|.
\end{align}
\end{Definition}

Then we can rewrite \cref{equ:pol1} as follows
\begin{equation}\label{equ:pol}
M_{p,q,r;\mu,\nu}^{\le,<,(2)}=\sum_{\Gamma}\prod\omega(e)F(\Gamma,\mu,\nu,p,q,r),
\end{equation}
where we sum over all reduced monodromy graphs of type $(0,\mu,\nu,p,q,r)$.\\
For the remainder of this subsection, we prove the following three claims constructively:
\begin{enumerate}
\item The set of reduced monodromy graphs of type $(0,\mu,\nu,p,q,r)$ only depends on the chamber $C$ (induced by the hyperplane arrangement $\mathcal{W}$) in which $\mu$ and $\nu$ are contained and not on the specific entries $\mu_i,\nu_j$. (\cref{lem:ch})
\item The product $\prod\omega(e)$ appearing in \cref{equ:pol} is a polynomial. (\cref{lem:weight})
\item The function $F(\Gamma,\nu,\nu)$ is a polynomial in each chamber. (\cref{lem:cou})
\end{enumerate}

Claim (1) was actually observed in \cite{CJM} for the case of (non-triply interpolated) monodromy graphs. For the convenience of the reader, we repeat the argument. We begin by introducing some Notation.

\begin{Notation}
Let $\mu$ be a partition and let $I\subset\{1,\dots,\ell(\mu)\}$. Then $\mu_I$ is the subpartition of $\mu$ given by $\mu_I=(\mu_{i_1},\dots,\mu_{i_{|I|}})$, where $i_j<i_{j+1}$.
\end{Notation}

\begin{Lemma}
\label{lem:ch}
The set of reduced monodromy graphs of type $(0,\mu,\nu,p,q,r)$ only depends on the chamber $C$.
\end{Lemma}

\begin{proof}
Let $\Gamma$ be a triply interpolated monodromy graph of type $(0,\mu,\nu,p,q,r)$, then we cut $\Gamma$ along $e$ and obtain two triply interpolated monodromy graphs $\Gamma_1$ and $\Gamma_2$. Let $e$ point away from $\Gamma_1$, then $\Gamma_1$ and $\Gamma_2$ are of respective type $(0,\mu_{I_1},\nu_{J_1}\cup\{\omega(e)\},k_1)$ and $(0,\mu_{I_2}\cup\{\omega(e)\},\nu_{J_2},k_2)$ for subsets $I_1,I_2\subset\{1,\dots,\ell(\mu)\}$ and $J_1,J_2\subset\{1,\dots,\ell(\nu)\}$.  and $k_1+k_2=k$. Moreover, we have $|\mu_{I_1}|=|\nu_{J_1}\cup\{\omega(e)\}|$ and we obtain 
\begin{equation}
\omega(e)=\sum_{i\in I_1}\mu_i-\sum_{j\in J_1}\nu_j.
\end{equation}
The only requirement for a reduced monodromy graph to contribute to the sum in \cref{equ:pol} is the positivity of all edge weights. As we saw above, this only depends on the chamber $C$ we pick.
\end{proof}

Claim (2) immediately follows:

\begin{Corollary}
\label{lem:weight}
Every edge weight $\omega(e)$ is a linear polynomial in the entries of $\mu$ and $\nu$. Thus $\prod\omega(e)$ is a polynomial in the entries of $\mu$ and $\nu$ as well.\qed
\end{Corollary}

Before we can prove claim (3), we need the following Definition.

\begin{Definition}
\label{def:chainpath}
Let $B$ a path in $\Gamma$ starting at an in-end, such that \begin{enumerate}
\item There are $s$ edges in $B$.
\item The first $s-1$ edges form a chain of bold edges. (see (8) of \cref{def:mono}
\item The last edge is dashed.
\end{enumerate}
We call $B$ a \textit{chain-path of length }$s$.
\end{Definition}

The following Lemma is our key step towards \cref{prop:piece}.

\begin{Lemma}
\label{lem:cou}
The function $F(\Gamma,\mu,\nu,p,q,r)$ can be expressed as a polynomial in each chamber $C$.
\end{Lemma}

\begin{proof}
We fix a reduced monodromy graph $\Gamma$ of type $(0,\mu,\nu,p,q,r)$. Assigning counters to $\Gamma$ translates to assigning counters to each chain-path in $\Gamma$ as follows:
Fix a chain-path $B$ of length $s$ and distribute the counter $l_{k}$ to the $k-$th edge $e_{k}$ in $B$. Moreover, let $B$ start at the in-end labeled $i$. Then $(l_1,\dots,l_s)$ satisfies the counter conditions if and only if\begin{enumerate}
\item $l_1\le\dots \le l_{s'}<\dots <\le l_s$ (see condition (13) in \cref{def:mono}) for some $t<s$,
\item $1\le l_1 \le \mu_i$ (see condition (14) in \cref{def:mono}),
\item $\max\{1,\mu_i-\omega(e_{k})\}\le l_{k}\le \mu_i$. (see condition (14) in \cref{def:mono})
\end{enumerate}
Thus we need to prove, that the cardinality of the set
\begin{equation}\label{equ:set}\{(l_1,\dots,l_s)|l_1\le\dots\le l_{s'}<\dots<l_s,\ l_1=1,\ \max\{1,\mu_i-\omega(e_{k})\}\le l_{k}\le \mu_i\}
\end{equation}
is piecewise polynomial in the entries of $\mu$ and $\nu$. We can express this cardinality as the following iterative sum
\begin{align}
\label{equ:chainsum}
\begin{split}
&\sum_{\substack{l_2=\\\max\{1,\mu_i-\omega(e_1)\}}}^{\mu_i}\dots\sum_{\substack{l_{s'}=\\\max\{l_{s-1},\mu_i-\omega(e_2)\}}}^{\mu_i}\sum_{\substack{l_{s+1}=\\\max\{l_{s-1}+1,\mu_i-\omega(e_2)\}}}^{\mu_i} \dots\\&\sum_{\substack{l_{s-1}=\\\max\{l_{s-2}+1,\mu_i-\omega(e_{s-1})\}}}^{\mu_i}\mu_i-\max\{l_{s-1},\mu_i-\omega(e_{s})\}.
\end{split}
\end{align}
If we know whether $\max\{\mu_i-\omega(e_1),1\}=\mu_i-\omega(e_1)$ and if we have a total ordering on the $\mu_i-\omega(e_{k})$, we can compute this sum using Faulhaber's formula \begin{equation}\sum_{k=1}^nk^p=\frac{1}{p+1}\sum_{j=0}^p(-1)^j\binom{p+1}{j}B_jn^{p+1-j},\end{equation} where $B_j$ is the $j-$th Bernoulli number. Notice, that the right hand side is a polynomial in $n$. Thus, the cardinality of the set in \cref{equ:set} is a polynomial in $\mu_i$ and the edge weights $\omega(e)$ (and since $\omega(e)$ is linear form in the entries of $\mu$ and $\nu$, the cardinality is a polynomial in the entries of $\mu$ and $\nu$), whenever we know the value of $\max\{\mu_1-\omega(e_1),1\}$ and if we have a total ordering on the $\mu_i-\omega(e_{k})$, which we can compute iteratively. Now, we show that choosing a chamber $C$ for $\mu$ and $\nu$ implies those conditions.\\
 Let $I_{k}\subset\{1,\dots,\ell(\mu)\}$ and $J_{k}\subset\{1,\dots,\ell(\nu)\}$ for $k=1,\dots,s$, such that
 \begin{equation}
 \omega(e_{k})=\sum_{j\in I_{k}}\mu_j-\sum_{j\in J_{k}}\nu_j
 \end{equation}
 for all $k$. We observe that for the edge $e_1$, we get \begin{equation}\max\{1,\mu_i-\omega(e_1)\}=\mu_i-\omega(e_1)\end{equation} if and only if \begin{equation}\sum_{j\in J_1} \nu_j-\sum_{j\in I_1-\{i\}}\mu_j>0.\end{equation} This implies that in a fixed chamber $C$, we know the value of $\max\{1,\mu_i-\omega(e_1)\}$. Moreover, we fix two edges $e_j$ and $e_{k}$, such that $j<k$. We see that since $e_j$ and $e_k$ are in the same chain-path and $e_{k}$ appears later than $e_j$, we have $I_j\subset I_{k}$ and $J_j\subset J_{k}$. Thus $\omega(e_{k})>\omega(e_j)$ if and only if \begin{equation}\sum_{l\in I_{k}-I_j}\mu_l-\sum_{l\in J_{k}-J_j}\nu_l>0.\end{equation}
 Thus we can answer whether $\omega(e_k)>\omega(e_j)$ in each chamber.\\
 Let $P_B(C)$ be the polynomial computing the cardinality of the set in \cref{equ:set} associated to the chain-path $B$ in the chamber $C$. Since we can choose counters in each chain-path independentely (they do not intersect, since chains of bold edges do not intersect), in the chamber $C$, the function $F(\Gamma,\mu,\nu,p,q,r)$ is given by $\prod P_B(C)$, where we take the product over all chain-paths. Since the graph is finite, $F(\Gamma,\mu,\nu,p,q,r)$ is a polynomial in the entries for $\mu$ and $\nu$ in each chamber $C$ as desired.
\end{proof}

We have now derived the following algorithm, which computes the polynomials $m_{p,q,r;\ell(\mu),\ell(\nu)}^{\le,<,(2)}(C)$ for $p+q+r=-2+\ell(\mu)+\ell(\nu)$, which also gives polynomials expressing $H_{p,q,r;\ell(\mu),\ell(\nu)}^{\le,<,(2)}$ in each chamber $C$ for $p=0$ or $q=0$.

\begin{Algorithm}
\label{alg:1}
\begin{algorithm}[H]{\textbf{Computing polynomials in genus $0$}}
\begin{algorithmic}[1]
\Procedure{Hurwitz}{$\ell(\mu),\ell(\nu),p,q,r,C$}\Comment{The polynomial expressing $M_{p,q,r;\mu,\nu}^{\le,<,(2)}$ and $H_{p,q,r;\mu,\nu}^{\le,<,(2)}$ in the chamber $C$ for $p=0$ or $q=0$}
\State $\mathfrak{G}(\ell(\mu),\ell(\nu),C,k)\gets$ Set of monodromy graphs of type $(0,\mu,\nu,p,q,r)$ in $C$
\ForAll{$\Gamma\in\mathfrak{G}(\ell(\mu),\ell(\nu),C,k)$}
\State $E(\Gamma)\gets$ Set of edges of $\Gamma$
\ForAll{$e\in E(\Gamma)$}
\State $\omega(e)\gets$ linear form in $\mu$ and $\nu$ expressing the weight of $e$
\EndFor
\State $W(\Gamma)\gets$ $\prod \omega(e)$, where the product is taken over all non-bold edges which are not adjacent to out-ends
\State $CP(\Gamma)\gets$ Set of all chain-paths in $\Gamma$
\ForAll{$P\in CP(\Gamma)$}
\State $q_P\gets$ Polynomial expressing \cref{equ:chainsum} in $C$
\EndFor
\EndFor
\State $C(\Gamma)\gets$ $\prod_{P\in CP(\Gamma)} q_P$\Comment{Polynomial obtained from the chain-paths}
\State $m(\Gamma)\gets$ The polynomial $W(\Gamma)\cdot C(\Gamma)$\Comment{Polynomial obtained from the edge weights}
\State $m_{p,q,r;\ell(\mu),\ell(\nu)}^{\le,<,(2)}(C)\gets$ $\sum_{\Gamma\in\mathfrak{G}(\ell(\mu),\ell(\nu),C,k)}m(\Gamma)$
\State \textbf{return} $m_{p,q,r;\ell(\mu),\ell(\nu)}^{\le,<,(2)}(C)$\Comment{The desired polynomial is $m_{p,q,r;\ell(\mu),\ell(\nu)}^{\le,<,(2)}(C)$}
\If{$p=0$ or $q=0$}
\State \textbf{return} $\frac{1}{\mu_1\cdots\mu_{\ell(\mu)}}m_{p,q,r;\ell(\mu),\ell(\nu)}^{\le,<,(2)}(C)$\Comment{Polynomial expressing $H_{p,q,r;\ell(\mu),\ell(\nu)}^{\le,<,(2)}$ in $C$}
\EndIf
\EndProcedure
\end{algorithmic}\qed
\end{algorithm}
\end{Algorithm}

\begin{Example}
\label{ex:c}
 We use this algorithm to compute the polynomials for $H_{0,2,0,\mu,\nu}^{\le,<,(2)}$ for $\ell(\mu)=\ell(\nu)=2$. The possible graphs are illustrated in \cref{fig:ex}. There are four chambers in that case as illustrated in Figure 9 in \cite{CJM}.\vspace{\baselineskip}\\
 We start with the chamber $C_1$ given by $\mu_1>\nu_1,\mu_1>\nu_2,\mu_2<\nu_1,\mu_2<\nu_2$. In this chamber the graphs $\mathrm{I.a,I.b,IV,V,VI,VII}$ contribute positive multiplicities:

\begin{align}
   &\mathrm{mult}(\mathrm{I.a})=\mu_1\left(\frac{\mu_2^2}{2}+\frac{\mu_2}{2}\right),\\
   &\mathrm{mult}(\mathrm{I.b})=\mu_1\left(\frac{\mu_2^2}{2}+\frac{\mu_2}{2}\right),\\
   &\mathrm{mult}(\mathrm{IV})=(\mu_1-\nu_2)\mu_2(\mu_1-\nu_2),\\
   &\mathrm{mult}(\mathrm{V})=(\mu_1-\nu_1)\mu_2(\mu_1-\nu_1),\\
   &\mathrm{mult}(\mathrm{VI})=(\mu_1-\nu_2)\mu_2\nu_2,\\
   &\mathrm{mult}(\mathrm{VII})=(\mu_1-\nu_1)\mu_2\nu_1.\\
\intertext{Adding all these contributions we obtain}
&m_{0,2,0,\mu,\nu}^{\le,<,(2)}(C_1)=\mu_1\mu_2(2\mu_1+\mu_2-\nu_1-\nu_2+1)=\mu_1\mu_2(\mu_1+1).
\intertext{Next we look at the chamber $C_2$ given by $\mu_1<\nu_1,\mu_1>\nu_2,\mu_2<\nu_1,\mu_2>\nu_2$. In this chamber the graphs $\mathrm{I.a,I.b,IV,V}$ contribute positive multiplicities:}
  &\mathrm{mult}(\mathrm{I.a})=\mu_1\left(\frac{\mu_2^2}{2}+\frac{\mu_2}{2}\right),\\
  &\mathrm{mult}(\mathrm{I.b})=\mu_1\left(\mu_2\nu_2-\frac{\nu_2^2}{2}+\frac{\nu_2}{2}\right),\\
  &\mathrm{mult}(\mathrm{III})=\mu_1\left(\frac{\mu_2^2}{2}+\frac{\mu_2}{2}-\mu_2\nu_2+\frac{\nu_2^2}{2}-\frac{\nu_2}{2}\right),\\
  &\mathrm{mult}(\mathrm{IV})=(\mu_1-\nu_2)\mu_2\nu_2,\\
  &\mathrm{mult}(\mathrm{VI})=(\mu_1-\nu_2)\mu_2(\mu_1-\nu_2).
\intertext{Adding all these contributions we obtain}
&m_{0,2,0,\mu,\nu}^{\le,<,(2)}(C_2)=\mu_1\mu_2\left(\nu_1+1\right).
\intertext{Let the chamber $C_3$ be given by $\mu_1<\nu_1,\mu_1<\nu_2,\mu_2>\nu_1,\mu_2>\nu_2$. In this chamber the graphs $\mathrm{I.a,I.b,II,III}$ contribute positive multiplicities:}
  &\mathrm{mult}(\mathrm{I.a})=\mu_1\left(\mu_2\nu_1-\frac{\nu_1^2}{2}+\frac{\nu_1}{2}\right),\\
  &\mathrm{mult}(\mathrm{I.b})=\mu_1\left(\mu_2\nu_2-\frac{\nu_2^2}{2}+\frac{\nu_2}{2}\right),\\
  &\mathrm{mult}(\mathrm{II})=\mu_1\left(\frac{\mu_2^2}{2}+\frac{\mu_2}{2}-\mu_2\nu_1+\frac{\nu_1^2}{2}-\frac{\nu_1}{2}\right),\\
  &\mathrm{mult}(\mathrm{III})=\mu_1\left(\frac{\mu_2^2}{2}+\frac{\mu_2}{2}-\mu_2\nu_2+\frac{\nu_2^2}{2}-\frac{\nu_2}{2}\right).
\intertext{Adding all these contributions we obtain}
&m_{0,2,0,\mu,\nu}^{\le,<,(2)}(C_3)=\mu_1\mu_2(\mu_2+1).
\intertext{Lastly, we consider the chamber $C_4$ given by $\mu_1>\nu_1,\mu_1<\nu_2,\mu_2>\nu_1,\mu_2<\nu_2$. In this chamber the graphs $\mathrm{I.a,I.b,II,V}$ contribute positive multiplicities:}
  &\mathrm{mult}(\mathrm{I.a})=\mu_1\left(\mu_2\nu_1-\frac{\nu_1^2}{2}+\frac{\nu_1}{2}\right),\\
  &\mathrm{mult}(\mathrm{I.b})=\mu_1\left(\frac{\mu_2^2}{2}+\frac{\mu_2}{2}\right),\\
  &\mathrm{mult}(\mathrm{II})=\mu_1\left(\frac{\mu_2^2}{2}+\frac{\mu_2}{2}-\mu_2\nu_1+\frac{\nu_1^2}{2}-\frac{\nu_1}{2}\right),\\
  &\mathrm{mult}(\mathrm{V})=(\mu_1-\nu_1)\mu_2\nu_2,\\
  &\mathrm{mult}(\mathrm{VII})=(\mu_1-\nu_1)\mu_2(\mu_1-\nu_1).
 \intertext{Adding all these contributions we obtain}
 &m_{0,2,0,\mu,\nu}^{\le,<,(2)}(C_4)=\mu_1\mu_2\left(\nu_2+1\right).
\end{align}
Thus we see 
\begin{align} 
 H_{0,2,0,\mu,\nu}^{\le,<,(2)}(C_1)=\mu_1+1,\ &\ H_{0,2,0,\mu,\nu}^{\le,<,(2)}(C_2)=\nu_1+1\\
 H_{0,2,0,\mu,\nu}^{\le,<,(2)}(C_3)=\mu_2+1,\ &\ H_{0,2,0,\mu,\nu}^{\le,<,(2)}(C_4)=\nu_2+1
 \end{align}
\end{Example}

\begin{figure}
\centering
\begin{tikzpicture}
\draw [dashed] (0,0)--(1,1);
\draw [line width=0.6mm] (0,2)--(1,1);
\draw [line width=0.6mm] (1,1)--(2,1);
\draw [line width=0.6mm] (2,1)--(3,2);
\draw (2,1)--(3,0);
\draw (0,2) node[anchor=east] {$\mu_2$};
\draw (0,0) node[anchor=east] {$\mu_1$};
\draw (3,2) node[anchor=west] {$\nu_1$};
\draw (3,0) node[anchor=west] {$\nu_2$};
\draw (-1,1) node[anchor=east] {$I.a$};
\end{tikzpicture}\begin{tikzpicture}
\draw [dashed] (0,0)--(1,1);
\draw [line width=0.6mm] (0,2)--(1,1);
\draw [line width=0.6mm] (1,1)--(2,1);
\draw [line width=0.6mm] (2,1)--(3,2);
\draw (2,1)--(3,0);
\draw (0,2) node[anchor=east] {$\mu_2$};
\draw (0,0) node[anchor=east] {$\mu_1$};
\draw (3,2) node[anchor=west] {$\nu_2$};
\draw (3,0) node[anchor=west] {$\nu_1$};
\draw (-1,1) node[anchor=east] {$I.b$};
\end{tikzpicture}\begin{tikzpicture}
\draw [line width=0.6mm] (0,2)--(1,2);
\draw (1,2)--(3,2);
\draw [line width=0.6mm] (1,2)--(2,1);
\draw [dashed] (0,0)--(2,1);
\draw [line width=0.6mm](2,1)--(3,0);
\draw (0,2) node[anchor=east] {$\mu_2$};
\draw (0,0) node[anchor=east] {$\mu_1$};
\draw (3,2) node[anchor=west] {$\nu_1$};
\draw (3,0) node[anchor=west] {$\nu_2$};
\draw (-1,1) node[anchor=east] {$II$};
\end{tikzpicture}

\begin{tikzpicture}
\draw [line width=0.6mm] (0,2)--(1,2);
\draw (1,2)--(3,2);
\draw [line width=0.6mm] (1,2)--(2,1);
\draw [dashed] (0,0)--(2,1);
\draw [line width=0.6mm](2,1)--(3,0);
\draw (0,2) node[anchor=east] {$\mu_2$};
\draw (0,0) node[anchor=east] {$\mu_1$};
\draw (3,2) node[anchor=west] {$\nu_2$};
\draw (3,0) node[anchor=west] {$\nu_1$};
\draw (-1,1) node[anchor=east] {$III$};
\end{tikzpicture}\begin{tikzpicture}
\draw [line width=0.6mm] (0,2)--(2,2);
\draw [line width=0.6mm] (2,2)--(3,2);
\draw [line width=0.6mm] (0,0)--(1,0);
\draw (1,0)--(3,0);
\draw [dashed] (1,0)--(2,2);
\draw (0,2) node[anchor=east] {$\mu_2$};
\draw (0,0) node[anchor=east] {$\mu_1$};
\draw (3,2) node[anchor=west] {$\nu_1$};
\draw (3,0) node[anchor=west] {$\nu_2$};
\draw (-1,1) node[anchor=east] {$IV$};
\end{tikzpicture}\begin{tikzpicture}
\draw [line width=0.6mm] (0,2)--(2,2);
\draw [line width=0.6mm] (2,2)--(3,2);
\draw [line width=0.6mm] (0,0)--(1,0);
\draw (1,0)--(3,0);
\draw [dashed] (1,0)--(2,2);
\draw (0,2) node[anchor=east] {$\mu_2$};
\draw (0,0) node[anchor=east] {$\mu_1$};
\draw (3,2) node[anchor=west] {$\nu_2$};
\draw (3,0) node[anchor=west] {$\nu_1$};
\draw (-1,1) node[anchor=east] {$V$};
\end{tikzpicture}

\begin{tikzpicture}
\draw [line width=0.6mm] (0,2)--(2,2);
\draw [line width=0.6mm] (2,2)--(3,2);
\draw [line width=0.6mm](0,0)--(1,0);
\draw (1,0)--(3,0);
\draw [dashed] (1,0)--(2,2);
\draw (0,2) node[anchor=east] {$\mu_2$};
\draw (0,0) node[anchor=east] {$\mu_1$};
\draw (3,2) node[anchor=west] {$\nu_2$};
\draw (3,0) node[anchor=west] {$\nu_1$};
\draw (-1,1) node[anchor=east] {$VI$};
\end{tikzpicture}\begin{tikzpicture}
\draw [line width=0.6mm] (0,2)--(2,2);
\draw [line width=0.6mm] (2,2)--(3,2);
\draw [line width=0.6mm](0,0)--(1,0);
\draw (1,0)--(3,0);
\draw [dashed] (1,0)--(2,2);
\draw (0,2) node[anchor=east] {$\mu_2$};
\draw (0,0) node[anchor=east] {$\mu_1$};
\draw (3,2) node[anchor=west] {$\nu_1$};
\draw (3,0) node[anchor=west] {$\nu_2$};
\draw (-1,1) node[anchor=east] {$VII$};
\end{tikzpicture}
  \caption{The graphs appearing for $(0,\mu,\nu,2,0,0)$ for $\ell(\mu)=\ell(\nu)=2$.}
 \label{fig:ex}
\end{figure}
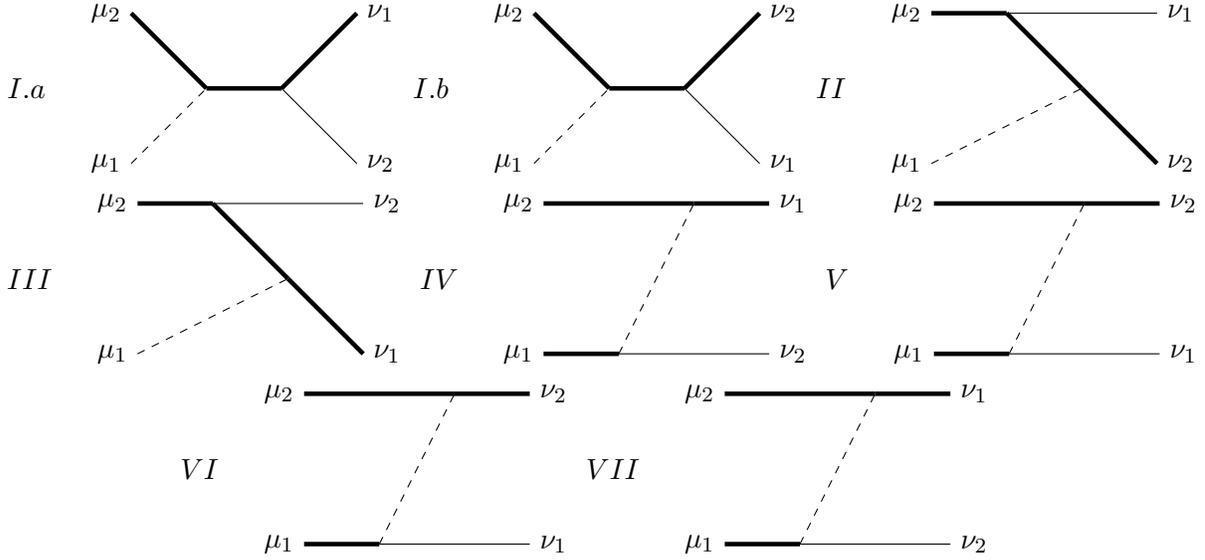

\subsubsection{Piecewise polynomiality in arbitrary genus}
\label{sec:pi1}
For higher genera, we once again study 
\begin{equation}
M_{p,q,r;\mu,\nu}^{\le,<,(2)}=\sum m(\Gamma).
\end{equation}
For the rest of this subsection we let $g,p,q,r$ be non-negative integers, such that $p+q+r=2g-2+\ell(\mu)+\ell(\nu)$ for the considered partitions $\mu,\nu$. Let $\Gamma$ be a triply interpolated monodromy graph of type $(g,\mu,\nu,p,q,r)$. We introduce a variable $x_i$ for each of the $g$ cycles in $\Gamma$. Then by a similar argument as before, each edge weight may be expressed as a linear polynomial in the entries of $\mu$ and $\nu$ and in the $x_i$. We still require all edge weights $\omega(e)$ to be greater than zero and thus obtain a hyperplane arrangement $H$ in the entries of $\mu$ and $\nu$ and $x_i$. We will refine this hyperplane arrangement along the way and obtain the piecewise polynomiality result that way.\\
As in the genus $0$ case, we rewrite the equation above by passing over to reduced monodromy graphs:
\begin{equation}
\label{equ:gen}
M_{p,q,r;\mu,\nu}^{\le,<,(2)}=\sum_{\Gamma}\prod\omega(e)F(\Gamma,\mu,\nu,p,q,r),
\end{equation}
where we sum over all reduced monodromy graphs of type $(g,\mu,\nu,p,q,r)$. Since there are additional variables $x_i$ the order on each chain-path is no longer determined just by the entries of $\mu$ and $\nu$. Thus, we need to refine the sum further. We begin by restricting to chambers due to a generalisation of \cref{lem:ch} which was proved in Theorem 3.9 in \cite{CJMa}:

\begin{thm}[\cite{CJMa}]
The set of reduced monodromy graphs of type $(g,\mu,\nu,p,q,r)$ is the same for all $\mu,\nu$ in the same chamber $C$.
\end{thm}

Thus, we know that $\prod\omega(e)$ is a polynomial in the $x_i$ and the entries of $\mu$ and $\nu$ in each chamber. Now, we want to express $F(\Gamma,\mu,\nu,p,q,r)$ as a polynomial as well. They key point in the proof of \cref{lem:cou} was the fact that the graph structure imposed an ordering on the edge weights $\omega(e)$ and $1$. For higher genera this is no longer true as depending on the values of $x_i$ there may be several orderings on each chain-path (see \cref{ex:g}). To deal with this problem, we introduce the notion of an ordering on a reduced monodromy graph.

\begin{Definition}
An ordering $O$ on a reduced monodromy graph $\Gamma$ is a partial ordering on the edge weights and $1$, that restricted to each chain-path and $1$ is a total ordering. We denote by $\mathcal{O}(\Gamma)$ the possible orderings on $\Gamma$.
\end{Definition}

Next, we refine the function $F(\Gamma,\mu,\nu,p,q,r)$.
\begin{Definition}
Let $\Gamma$ be a reduced monodromy graph and $O$ an ordering on $\Gamma$. Then we define $F(\Gamma,\mu,\nu,p,q,r,\underline{x},O)$ (where $\underline{x}=x_1,\dots,x_g$) to be the function counting all possible counter distributions on $\Gamma$ compatible with $O$.
\end{Definition}
We want to argue that $F(\Gamma,\mu,\nu,p,q,r,\underline{x},O)$ is a polynomial in the $x_i$ and the entries of $\mu$ and $\nu$. However, we have to be careful about the values of $\underline{x}$, since not all choices of $\underline{x}$ are compatible with $O$. Thus, we define $Q(\Gamma,\mu,\nu,O)$ to be the set of all values for $x_i$ fulfilling the ordering $O$. It is easy to see, that this set is convex and the $x_i$ are bounded since all edge weights have to be positive. We thus obtain the following Lemma:

\begin{Lemma}
The set $Q(\Gamma,\mu,\nu,O)$ is a polytope with equations given by linear forms in the entries of $\mu$ and $\nu$.
\end{Lemma}

\begin{Definition}
We denote the hyperplane arrangement in $\mathbb{N}^{\ell(\mu)+\ell(\nu)}$ induced by the combinatorial types of $Q(\Gamma,\mu,\nu,O)$ by $\mathcal{V}(\Gamma,O)$
\end{Definition}

By same argument as in \cref{lem:cou}, we also get the following Lemma:

\begin{Lemma}
The function $F(\Gamma,\mu,\nu,p,q,r,\underline{x},O)$ is a polynomial in $\underline{x}$ and the entries of $\mu$ and $\nu$ for $\underline{x}\in Q(\Gamma,\mu,\nu,O)$.
\end{Lemma}

We can now rewrite \cref{equ:gen} as follows:

\begin{equation}
M_{p,q,r;\mu,\nu}^{\le,<,(2)}=\sum_{\Gamma}\sum_{O\in\mathcal{O}(\Gamma)}\sum_{\underline{x}\in Q(\Gamma,\mu,\nu,O)}\prod\omega(e)F(\Gamma,\mu,\nu,p,q,r,\underline{x},O).
\end{equation}
It is well-known that summing a polynomial over a polytope with rational vertices yields a quasi-polynomial (see e.g. \cite{zbMATH06340142}, \cite{baldoni2014three}).\\
Since $\prod\omega(e)F(\Gamma,\mu,\nu,p,q,r,\underline{x},O)$ is a polynomial in $\underline{x}$ and the entries of $\mu$ and $\nu$ and since $Q(\Gamma,\mu,\nu,O)$ is a polytope, $M_{p,q,r;\mu,\nu}^{\le,<,(2)}$ is a quasi-polynomial in each chamber of the hyperplane arrangement given as the common refinement of $\mathcal{W}$ and the family $(\mathcal{V}(\Gamma,O))_{\Gamma,O}$.

\begin{Remark}
\label{re:poly}
We note that with our method, we only proved that we obtain piecewise quasi-polynomiality for $H_{p,q,r;\ell(\mu),\ell(\nu)}^{\le,<,(2)}$. However, we know by \cref{thm:gjv} that the triply interpolated Hurwitz number is a polynomial for $p=0$. Moreover, in \cref{alg:2} we pick one chamber $C'$ induced by the refined hyperplane arrangement in $C$. However, by \cref{thm:gjv} the result does not depend on the choice of the finer chamber in $C$ for $q=0$. By \cite{hahn2018wall} the same is true for $p=0$.
\end{Remark}

We can now state our algorithm for higher genera, which computes the polynomials for $H_{p,q,r;\ell(\mu),\ell(\nu)}^{\le,<,(2)}$ when $p=0$ or $q=0$..

\begin{Algorithm}
\label{alg:2}
\begin{algorithm}[H]{\textbf{Computing polynomials in genus $g$}}
\begin{algorithmic}[1]
\Procedure{Hurwitz}{$\ell(\mu),\ell(\nu),p,q,r,C$}\Comment{The polynomial expressing $M_{p,q,r;\mu,\nu}^{\le,<,(2)}$ and $H_{p,q,r;\mu,\nu}^{\le,<,(2)}$ in the chamber $C$ for $p=0$ or $q=0$}
\State  $\mathfrak{G}(\ell(\mu),\ell(\nu),C,g,k)\gets$ Set of all monodromy graphs of type $(g,\mu,\nu,p,q,r)$ in $C$
\ForAll{$\Gamma\in\mathfrak{G}(\ell(\mu),\ell(\nu),C,g,k)$}
\State $\mathcal{O}(\Gamma)\gets$ Set of all orderings on $\Gamma$
\ForAll{$O\in\mathcal{O}(\Gamma)$}
\State $Q(\Gamma,\mu,\nu,O)\gets$ Polytope induced by the inequalities for $\Gamma$ in $C$ and $O$
\State $\mathcal{V}(\Gamma,O)\gets$ Hyperplane arrangement induced by the equations for $Q(\Gamma,\mu,\nu,O)$
\EndFor
\State $C(\mathcal{O})\gets$ Common refinement of $C$ and the family $\mathcal{V}(\Gamma,O)_{O\in\mathcal{O}(\Gamma)}$
\EndFor
\State $C(\mu,\nu)\gets$ Common refinement of $C$ and the family of chambers we computed above $(C(\mathcal{O}(\Gamma)))_{\Gamma\in\mathfrak{G}(\ell(\mu),\ell(\nu),C,g,k)}$
\State Choose some chamber $C'$ in $C(\mathcal{O}(\Gamma))$
\ForAll{$\Gamma\in\mathfrak{G}(\ell(\mu),\ell(\nu),C,g,k)$}
\State $E(\Gamma)\gets$ Set of all edges in $\Gamma$
\ForAll{$e\in E(\Gamma)$}
\State $\omega(e)\gets$ linear form in of $\mu$, $\nu$ and $\underline{x}$
\EndFor
\State $W(\Gamma)\gets\prod\omega(e)$, where we take the product over all non-bold edges $e\in E(\Gamma)$ which are not adjacent to out-ends
\State $CP(\Gamma)\gets$ State of all chain-paths
\ForAll{$O\in\mathcal{O}(\Gamma)$}
\ForAll{$P\in CP(\Gamma)$}
\State $q_P(O)\gets$ Polynomial expressing \cref{equ:chainsum} with respect to the order $O$
\EndFor
\State $c(O)\gets{\displaystyle \prod_{P\in CP(\Gamma)}g_P(O)}$
\EndFor
\State $m(\Gamma,C')\gets\sum_{O\in\mathcal{O}(\Gamma)}\sum_{\underline{x}\in Q(\Gamma,\mu,\nu,O)}\prod W(\Gamma)c(O)$
\EndFor
\State ${\displaystyle m_{p,q,r;\ell(\mu),\ell(\nu)}^{\le,<,(2)}(C)=\sum_{\Gamma\in\mathfrak{G}(\ell(\mu),\ell(\nu),C,g,k)}m(\Gamma,C')}$\Comment{The desired polynomial $m_{p,q,r;\ell(\mu),\ell(\nu)}^{\le,<,(2)}(C)$}
\State \textbf{return} $m_{p,q,r;\ell(\mu),\ell(\nu)}^{\le,<,(2)}(C)$
\If{$p=0$ or $q=0$}
\State \textbf{return} $\frac{1}{\mu_1\cdots\mu_{\ell(\mu)}}m_{p,q,r;\ell(\mu),\ell(\nu)}^{\le,<,(2)}(C)$ \Comment{Polynomial expressing $H_{p,q,r;\ell(\mu),\ell(\nu)}^{\le,<,(2)}$ in $C$}
\EndIf
\EndProcedure
\end{algorithmic}
\end{algorithm}
\end{Algorithm}

\begin{Example}
\label{ex:g}
In this Example, we treat the graph $\Gamma$ in \cref{fig:gg}. The weight function is \begin{equation}\mu_1\mu_2(\mu_2+\mu_3-x_2).\end{equation}
The only chain-path is given by $(\mu_3,\mu_2+\mu_3,x_1,\mu_1+x_1,\nu_1)$, thus the counter function is given by:
\begin{align}
F(\Gamma,\mu,\nu,3,0,0,\underline{x})=\sum_{l_2=1}^{\mu_3}\sum_{\substack{l_3=\max\\\{l_2,\mu_3-x_1+1\}}}^{\mu_3}\sum_{\substack{l_4=\max\\\{l_3,\mu_3-\mu_1-x_1+1\}}}^{\mu_3}\mu_3-l_4+1.
\end{align}
There are five different orderings:
\begin{align}
O_1:\nu_1>\mu_2+\mu_3>\mu_1+x_1>x_1>\mu_3\\
O_2:\nu_1>\mu_2+\mu_3>\mu_1+x_1>\mu_3>x_1\\
O_3:\nu_1>\mu_2+\mu_3>\mu_3>\mu_1+x_1>x_1\\
O_4:\nu_1>\mu_1+x_1>\mu_2+\mu_3>x_1>\mu_3\\
O_5:\nu_1>\mu_1+x_1>\mu_2+\mu_3>\mu_3>x_1
\end{align}
We show, how to compute the contributions for $O_1$ and $O_2$. For $O_1$ we obtain:
\begin{equation}
F(\Gamma,\mu,\nu,3,0,0,\underline{x},O_1)=\sum_{l_2=1}^{\mu_3}\sum_{\substack{l_3=l_2}}^{\mu_3}\sum_{\substack{l_4=l_3}}^{\mu_3}\mu_3-l_4+1.
\end{equation}
For the ordering $O_2$, we get the following formula:
\begin{equation}
F(\Gamma,\mu,\nu,3,0,0,\underline{x},O_1)=\sum_{l_2=1}^{\mu_3-x_1+1}\sum_{l_3=\mu_3-x_1}^{\mu_3}\sum_{l_4=l_3}^{\mu_3}\mu_3-l_4+1+\sum_{l_2=\mu_3-x_1+2}^{\mu_3}\sum_{l_3=l_2}^{\mu_3}\sum_{l_4=l_3}\mu_3-l_4+1.
\end{equation}
The ordering imposes the following inequalities on $x_1$:
\begin{align}
&O_1:\mu_2+\mu_3-\mu_1>x_1>\mu_3\\
&O_2:\min\{\mu_2+\mu_3-\mu_1,\mu_1\}>x_1>\max\{0,\mu_3-\mu_1\}\\
\end{align}
The inequality given by $O_2$ induces additional hyperplanes not given by equations of type $\sum:{i\in I}\mu_i-\sum_{j\in J}\nu_j$. The contributions of $O_1$ and $O_2$ (which we do not expand further, since the first sum alone expands to 20 terms) are
\begin{align}
&\sum_{x_1=\mu_3}^{\mu_2+\mu_3-\mu_1}\left(\mu_1\mu_2(\mu_2+\mu_3-x_2)\sum_{l_2=1}^{\mu_3}\sum_{\substack{l_3=l_2}}^{\mu_3}\sum_{\substack{l_4=l_3}}^{\mu_3}\mu_3-l_4+1\right)+\\
&\sum_{x_1=\max\substack\{0,\mu_3-\mu_1\}}^{\min\{\mu_2+\mu_3-\mu_1,\mu_1\}}\left(\mu_1\mu_2(\mu_2+\mu_3-x_2)\sum_{l_2=1}^{\mu_3-x_1+1}\sum_{l_3=\mu_3-x_1}^{\mu_3}\sum_{l_4=l_3}^{\mu_3}\mu_3-l_4+1+\right.\\&\left.\sum_{l_2=\mu_3-x_1+2}^{\mu_3}\sum_{l_3=l_2}^{\mu_3}\sum_{l_4=l_3}^{\mu_3}\mu_3-l_4+1\right),
\end{align}
which is a polynomial in each chamber of the refined hyperplane arrangement. We note that after computing the polynomial for every graph our method yields the same polynomial in each chamber (see \cref{re:poly}), while this may not be true of each graph.
\end{Example}

\begin{figure}
\begin{tikzpicture}
\draw [line width=0.6mm] (0,0)--(4,0) node[near end, anchor=south] {$\mu_2+\mu_3$};
\draw (4,0) .. controls(6,1) .. (8,0) node[midway, anchor=south] {$\mu_2+\mu_3-x_1$};
\draw [line width=0.6mm] (4,0) .. controls(6,-1) .. (8,0) node[near start, anchor=north east] {$x_1$} node[near end, anchor=north west] {$\mu_1+x_1$};
\draw[dashed] (8,0)--(10,0) node[anchor=west] {$\nu_1$};
\draw[dashed] (0,-1)--(2,0);
\draw[dashed] (0,-2)--(6,-0.75);
\draw node[anchor=east] {$\mu_3$};
\draw (0,-1) node[anchor=east] {$\mu_2$};
\draw (0,-2) node[anchor=east] {$\mu_1$};
\draw (0,-3)--(10,-3);
\draw[fill] (0,-3) circle (2pt) (0,-3) node[anchor=north] {$0$};
\draw[fill] (2,-3) circle (2pt) (2,-3) node[anchor=north] {$1$};
\draw[fill] (4,-3) circle (2pt) (4,-3) node[anchor=north] {$2$};
\draw[fill] (6,-3) circle (2pt) (6,-3) node[anchor=north] {$3$};
\draw[fill] (8,-3) circle (2pt) (8,-3) node[anchor=north] {$4$};
\draw[fill] (10,-3) circle (2pt) (10,-3) node[anchor=north] {$5$};
\end{tikzpicture}
\caption{Triply interpolated monodromy graph of genus $1$.}
\label{fig:gg}
\end{figure}
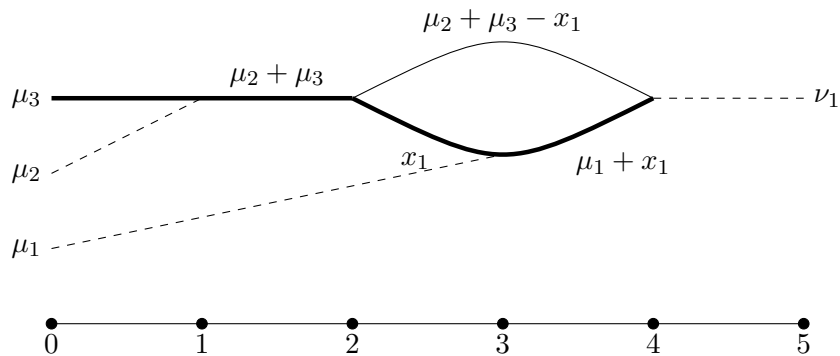

\subsection{Chamber behaviour in genus 0}
\label{sec:ch}
In this section, we define a counting problem in the symmetric group generalising triply interpolated Hurwitz numbers in genus $0$. We use this to obtain recursive wall-crossing formulae in genus $0$. As before, for fixed $\ell(\mu)$ and $\ell(\nu)$, let $m_{p,q,r;\ell(\mu),\ell(\nu)}^{\le,<,(2)}(C)$ be the polynomial expressing $M_{p,q,r;\ell(\mu),\ell(\nu)}^{\le,<,(2)}$ in the chamber $C$. Moreover, let $C_1$ and $C_2$ be adjacent chambers seperated by the wall $\delta=\sum_{i\in I}\mu_i-\sum_{j\in J}\nu_j$ (w.l.o.g $\delta>0$ in $C_1$). We want to compute the wall-crossing for fixed $\ell(\mu)$ and $\ell(\nu)$ \begin{equation}
WC_{C_1}^{C_2}(0,\ell(\mu),\ell(\nu),p,q,r)=m_{p,q,r;\ell(\mu),\ell(\nu)}^{\le,<,(2)}(C_1)-m_{p,q,r;\ell(\mu),\ell(\nu)}^{\le,<,(2)}(C_2).
\end{equation} 

We will study the wall-crossing formulae for $p=0$ or $q=0$, i.e. for interpolations between monotone and simple double Hurwitz numbers and interpolations between strictly monotone and simple double Hurwitz numbers.

\begin{Remark}
To study this wall-crossing problem, we define related Hurwitz-type counts generalising triply interpolated Hurwitz numbers. It is an interesting feature of this Hurwitz-type counting problem that the wall-crossing induced by the piecewise polynomial structure can itself be expressed in terms of these Hurwitz numbers with smaller input data. For a precise formulation, see \cref{def:general} and \cref{thm:wall}.
\end{Remark} 

\begin{disclaimer}
As mentioned before we will study the problem for $p=0$ or $q=0$. As the discussions are completely parallel, we will only work out the details for $q=0$ (i.e. the interpolation between monotone and simple double Hurwitz numbers).
\end{disclaimer}

We classify those reduced monodromy graphs $\Gamma$ having different multiplicity in $C_1$ than in $C_2$, since graphs with the same multiplicity cancel in $WC_{C_1}^{C_2}(0,\ell(\mu),\ell(\nu),p,0,r)$. By our discussion in \cref{sec:pi0}, there are five cases of graphs contributing to $m_{p,0,r;\ell(\mu),\ell(\nu)}^{\le,<,(2)}(C_1)$, which contribute a different multiplicity than $m_{p,0,r;\ell(\mu),\ell(\nu)}^{\le,<,(2)}(C_2)$:
 \begin{enumerate}
  \item [(1a)] The graphs contributing to $m_{p,0,r;\ell(\mu),\ell(\nu)}^{\le,<,(2)}(C_1)$ (resp. $m_{p,0,r;\ell(\mu),\ell(\nu)}^{\le,<,(2)}(C_2)$) having a normal edge of weight $\delta$ (resp. $-\delta$) emerging from one of the first $p$ vertices.
  \item [(1b)] The graphs contributing to $m_{p,0,r;\ell(\mu),\ell(\nu)}^{\le,<,(2)}(C_1)$ (resp. $m_{p,0,r;\ell(\mu),\ell(\nu)}^{\le,<,(2)}(C_2)$) having an non-normal edge of weight $\delta$ (resp. $-\delta$) emerging from one of the first $p$ vertices.
  \item [(1c)] The graphs contributing to $m_{p,0,r;\ell(\mu),\ell(\nu)}^{\le,<,(2)}(C_1)$ (resp. $m_{p,0,r;\ell(\mu),\ell(\nu)}^{\le,<,(2)}(C_2)$) having an edge of weight $\delta$ (resp. $-\delta$) emerging from one of the last $r$ vertices.
\item [(2a)] The graphs contributing to $m_{p,0,r;\ell(\mu),\ell(\nu)}^{\le,<,(2)}(C_1)$ (resp. $m_{p,0,r;\ell(\mu),\ell(\nu)}^{\le,<,(2)}(C_2)$) with a chain-path (see \cref{def:chainpath}) containing two edges $e$ and $e'$ ($e$ coming before $e'$), such that $\omega(e)-\omega(e')=\delta$ (resp. $(\omega(e)-\omega(e')=-\delta$).
\item [(2b)] The graphs contributing to $m_{p,0,r;\ell(\mu),\ell(\nu)}^{\le,<,(2)}(C_1)$ (resp. $m_{p,0,r;\ell(\mu),\ell(\nu)}^{\le,<,(2)}(C_2)$)  with a chain-path containing two edges $e$ and $e'$ ($e$ coming before $e'$), such that $\omega(e)-\omega(e')=-\delta$ (resp. $\omega(e)-\omega(e')=\delta$).
 \end{enumerate}
 
Conditions (1a), (1b) and (1c) correspond to the fact that every edge weight must be greater than $0$. Since $\delta<0$ in $C_2$, the graph $\Gamma$ has multiplicity $0$ in that chamber. Conditions (2a) and (2b) correspond to changes in the polynomials computing the counters for each chain-path: The polynomial can change if $\mu_i-\omega(e)$ ($e$ contained in a chain-path starting at $\mu_i$) or $\omega(e)-\omega(e')$ ($e$ and $e'$ contained in the same chain-path) changes sign by crossing $\delta$. Note that $\mu_i-\omega(e)=\omega(e')-\omega(e)$ for $e'$ being the in-end $\mu_i$. Thus, we obtain the two cases (2a) and (2b).\\
The following idea will be our main tool in this section: We start with a triply interpolated monodromy graph and cut it along some distinguished edge (resp. two distinguished edges). Since the graphs are of genus $0$, we obtain two (resp. three) new triply interpolated monodromy graphs. As a first step, we classify the pairs (resp. triples) of graphs we can obtain by this cutting process. Our second step is a regluing process. We will glue our graphs from two (resp. three) smaller graphs. The key observation here is, that in order to obtain a triply interpolated monodromy graph again, this gluing has to respect the following:
\begin{enumerate}
\item The ordering of the chains of bold edges.
\item The monotonicity of the counters, if we glue edges to a new chain-path.
\end{enumerate}
In order to formalise this, we introduce a new and more general Hurwitz-type counting problem, where these two conditions are framed in terms of what we call \textit{start and end conditions}. In some sense, these start and end conditions store the information concerning the counter and position of the edge, which we cut in the first place. Analysing this regluing process, we obtain a recursive wall-crossing formula for this more general counting problem. To understand the general idea, we start by decomposing the graphs above into smaller graphs and thus make the mentioned cutting process more precise.
\begin{enumerate}
\item [(1a)] Let $\Gamma$ be a triply interpolated monodromy graph of type $(0,\mu,\nu,p,0,r)$ with a normal edge of weight $\delta=\sum_{i\in I}\mu_i-\sum_{j\in J}\nu_j$, emanating from one of the first $p+q$ vertices. We cut the graph along the edge $\delta$ and obtain two graphs $\Gamma_1$ and $\Gamma_2$ of respective type $(0,\mu_I,(\nu_J,\delta),p_1,0,r_1)$ and $(0,(\mu_{I^c},\delta),\nu_{J^c},p_2,0,r_2)$ with $p_1+p_2=p$, and $r_1+r_2=r$. Starting with two triply interpolated monodromy graphs $\Gamma_1$ and $\Gamma_2$ of respective type $(0,\mu_I,(\nu_J,\delta),p_1,0,r_1)$ and $(0,(\mu_{I^c},\delta),\nu_{J^c},p_2,0,r_2)$, we want to glue them along the edge corresponding to $\delta$. However, this does not always yield a triply interpolated monodromy graph (e.g: If $\delta$ is a normal edge in $\Gamma_1$, but a bold edge in $\Gamma_2$.) Thus, we need some compatibility condition for these graphs. In fact, in order to obtain a triply interpolated monodromy graph of type $(0,\mu,\nu,p,0,r)$ with a normal edge of weight $\delta$, the edge $\delta$ must be normal in $\Gamma_1$ and dashed in $\Gamma_2$. Furthermore, if $\delta$ emanates from a chain of bold edges starting at $\mu_l$ in $\Gamma_1$ the edge $\delta$ must join with a chain of bold edges starting at $\mu_j$ in $\Gamma_2$ where $j>l$ (since the dashed and normal edges connect chains of bold edges). This corresponds to the end condition of type $(1,l,i)$ for $\Gamma_1$ and the start condition of type $(1,l)$ for $\Gamma_2$ in \cref{def:gen}, where $i$ is the label we choose for the out-end of $\Gamma_1$ corresponding to $\delta$.
\item [(1b)] Let $\Gamma$ be a triply interpolated monodromy graph of type $(0,\mu,\nu,p,0,r)$ with a non-normal edge of weight $\delta=\sum_{i\in I}\mu_i-\sum_{j\in J}\nu_j$. As before, we cut along $\delta$ and obtain two graphs $\Gamma_1$ and $\Gamma_2$ of respective type $(0,\mu_I,(\nu_J,\delta),p_1,0,r_1)$ and $(0,(\mu_{I^c},\delta),\nu_{J^c},p_2,0,r_2)$. Gluing two graphs $\Gamma_1$ and $\Gamma_2$ of these types along $\delta$, we see that in order to obtain a graph as in (1b), there are two types of conditions: Either $\delta$ is dashed in both $\Gamma_1$ and $\Gamma_2$ and if $\delta$ is contained in a chain-path starting $\mu_l$ in $\Gamma_1$, the in-end $\delta$ in $\Gamma_2$ must join with a chain of bold edges starting at $\mu_j$ in $\Gamma_2$ with $j>l$. Alternatively, $\delta$ is dashed in $\Gamma_1$ and bold in $\Gamma_2$. Morever if $\delta$ has counter $c$ in $\Gamma_1$, the first inner edge of the chain of bold edges starting at $\delta$ in $\Gamma_2$ must have counter $c'>c$. This corresponds to the end condition of type $(2,i,l,c)$ for $\Gamma_1$ and start condition of type $(2,l,c)$ in \cref{def:gen}, where $i$ is the label we choose for the out-end of $\Gamma_1$ corresponding to $\delta$.
\item [(1c)] Let $\Gamma$ be a triply interpolated monodromy graph of type $(0,\mu,\nu,p,0,r)$ with a normal edge of weight $\delta=\sum_{i\in I}\mu_i-\sum_{j\in J}\nu_j$, emanating from one of the last $r$ vertices. As before, we cut along $\delta$ and obtain two graphs $\Gamma_1$ and $\Gamma_2$ of respective type $(0,\mu_I,(\nu_J,\delta),p_1,0,r_1)$ and $(0,(\mu_{I^c},\delta),\nu_{J^c},p_2,0,r_2)$. Here the only condition for the gluing process we require is $\delta$ only interacting with one of the last $r_j$ vertices in $\Gamma_j$ ($j=1,2$). This corresponds to the end and start condition $(3,i)$ in \cref{def:gen}, where $i$ is the label we choose for the out-end of $\Gamma_1$ corresponding to $\delta$. 
\item [(2a)] We once again start with a triply interpolated monodromy graph of type $(0,\mu,\nu,p,0,r)$. We impose the condition that there is a chain-path with two edges $e_1,e_2$ ($e_1$ appearing before $e_2$), such that
\begin{equation}
\omega(e_1)-\omega(e_2)=\delta=\sum_{i\in I}\mu_i-\sum_{j\in J}\nu_j.
\end{equation}
For
\begin{equation}\label{equ:e1}
\omega(e_1)=\sum_{i\in I_1}\mu_i-\sum_{j\in J_1}\nu_j
\end{equation}
and
\begin{equation}\label{equ:e2}
\omega(e_2)=\sum_{i\in I_2}\mu_i-\sum_{j\in J_2}\nu_j,
\end{equation}
and $\omega(e_1)-\omega(e_2)=\delta$ this translates to $I_2=I_1\sqcup I^c$ and $J_2=J_1\sqcup J^c$. Cutting along the edges $e_1$ and $e_2$, we obtain three graphs $\Gamma_1$, $\Gamma_2$ and $\Gamma_3$ of respective types
\begin{align}
\label{equ:type1}&(0,\mu_{I_1},\left(\nu_{J_1},\sum_{i\in I_1}\mu_i-\sum_{j\in J_1}\nu_j\right),p_1,0,r_1),\\
\label{equ:type2}&(0,\left(\mu_{I^c},\sum_{i\in I_1}\mu_i-\sum_{j\in J_1}\nu_j\right),\left(\nu_{J^c},\sum_{i\in I_2}\mu_i-\sum_{j\in J_2}\nu_j\right),p_2,0,r_2),\\
\label{equ:type3}&(0,\left(\mu_{I_2^c},\sum_{i\in I_2}\mu_i-\sum_{j\in J_2}\nu_j\right),\nu_{J_2^c},p_3,0,r_3),
\end{align}
where $p_1+p_2+p_3=p$ and $r_1+r_2+r_3=r$.
Regluing graphs of these respective types corresponds to the gluing process in (1b). Thus we need an end condition of type $(2,l,c,i)$ for $\Gamma_1$, start condition of type $(2,l,i)$ for $\Gamma_2$, end condition of type $(2,l,c,i)$ for $\Gamma_2$ and start condition of type $(2,l,i)$ for $\Gamma_3$. (If $I_1={p}$ and $J_1=\emptyset$, we only cut at $\delta+\mu_p$ and thus obtain only $\Gamma_2$ and $\Gamma_3$.)
\item [(2b)] Starting with a triply interpolated monodromy graph of type $(0,\mu,\nu,p,0,r)$, with a chain-path containing two edges $e_1$ and $e_2$ (with $e_1$ appearing before $e_2$, such that
\begin{equation}\omega(e_1)-\omega(e_2)=-\delta=\sum_{i\in I^c}\mu_i-\sum_{j\in J^c}\nu_j.
\end{equation}
For $\omega(e_1)$ and $\omega(e_2)$ as in \cref{equ:e1} and \cref{equ:e2} respectively, and $\omega(e_1)-\omega(e_2)=\delta$ this translates to $I_2=I_1\sqcup I$ and $J_2=J_1\sqcup J$. Similarly as in (2a), we cut along $e_1$ and $e_2$ to obtain three graphs $\Gamma_1$, $\Gamma_2$ and $\Gamma_3$ of respective types
\begin{align}&(0,\mu_{I_1},\left(\nu_{J_1},\sum_{i\in I_1}\mu_i-\sum_{j\in J_1}\nu_j\right),p_1,0,r_1),\\&(0,\left(\mu_{I},\sum_{i\in I_1}\mu_i-\sum_{j\in J_1}\nu_j\right),\left(\nu_{J},\sum_{i\in I_2}\mu_i-\sum_{j\in J_2}\nu_j\right),p_2,0,r_2),\\&(0,\left(\mu_{I_2^c},\sum_{i\in I_2}\mu_i-\sum_{j\in J_2}\nu_j\right),\nu_{J_2^c},p_3,0,r_3),
\end{align}
where $p_1+p_2+p_3=p$ and $r_1+r_2+r_3=r$. Regluing graphs of these respective types, we need to impose the same end and start conditions as in (2a).
\end{enumerate}

\begin{Notation}
 We fix two partitions $\mu$ and $\nu$.
 \begin{enumerate}
\item For a subset $I=\{i_1,\dots,i_n\}$ (where $i_j<i_{j+1}$) of $\{1,\dots,\ell(\mu)\}$ and positive integers $\delta$, $j$ ($j\notin I$), we denote by $(\mu_I,\delta)_j$ the partition $(\mu_{i_1},\dots,\mu_{i_j},\delta,\mu_{i_j+1},\dots,\mu_{i_n})$.
 \item Let $(\sigma_1,\tau_1,\dots,\tau_r,\sigma_2)$ be a triply interpolated factorisation with $\mathcal{C}(\sigma_1)=\mu$ and $\mathcal{C}(\sigma_2)=\nu$. We define $\tau^1(l)$ to be the transposition with the biggest position containing elements of the cycle of $\sigma_2$ labeled $l$. Moreover, let $t^1(l)$ be the position of $\tau^1(l)$.
 \item Let $(\sigma_1,\tau_1,\dots,\tau_r,\sigma_2)$ be a triply interpolated factorisation with $\mathcal{C}(\sigma_1)=\mu$ and $\mathcal{C}(\sigma_2)=\nu$. We define $\tau^2(l)$ to be the transposition with the smallest position containing elements of the cycle of $\sigma_1$ labeled $l$. Morover, let $t^2(l)$ be the position of $\tau^2(l)$.
 \end{enumerate}
\end{Notation}

\begin{Definition}
\label{def:gen}
Let $\mu,\nu,p,q,r$ be data as before. Let $\eta=(\sigma_1,\tau_1,\dots,\tau_m,\sigma_2)$ be a triply interpolated factorisation of type $(0,\mu,\nu,p,0,r)$, such that $\tau_i=(r_i\ s_i)$.\\
We begin by defining end conditions:
\begin{enumerate}
\item We say $\eta$ satisfies end condition $(1,l,i)$ if 
\begin{itemize}
\item $t^1(i)\le k$
\item $\sum_{j=1}^{l-1}\mu_j+1\le s_{t^1(i)}\le \sum_{k=1}^l\mu_i$
\end{itemize}
In monodromy graph language, this corresponds to the following picture: The out-end corresponding to $\nu_i$ is coloured normal and emanates from the chain of bold edges starting at $\mu_l$ (see \cref{fig:case1}).\vspace{\baselineskip}\\
\begin{figure}[ht]
\begin{center}
\begin{tikzpicture}
\draw [line width=0.6mm] (0,0)--(4,0);
\draw [line width=0.6mm] (4,0)--(6,1);
\draw [line width=0.6mm] (6,1)--(7,1);
\draw (6,1)--(7,0);
\draw (2,0)--(1,-1);
\draw (4,0)--(5,-1);
\draw (0,-2)--(7,-2);
\draw (0,0) node[anchor=east] {$\mu_l$};
\draw (7,0) node[anchor=west] {$\nu_i$};
\draw[fill] (6,-2) circle (2pt) node[anchor=north] {$t^1(i)$};
\end{tikzpicture}
\caption{Schematic drawing of end condition $(1,l,i)$.}
\label{fig:case1}
\end{center}
\end{figure}
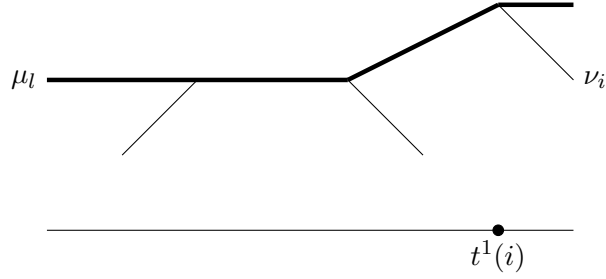
\item We say $\eta$ satisfies end condition $(2,l,c,i)$ if 
\begin{itemize}
\item $t^1(i)\le k$
\item $s_{t^1(i)}=\sum_{j=1}^{l-1}\mu_j+c$ for $0\le c<\mu_l$
\end{itemize}
In monodromy graph language, this corresponds to the following picture: The out-end corresponding to $\nu_i$ is coloured dashed, has counter $c$ and emanates from the chain of bold edges starting at $\mu_l$. (See \cref{fig:case3}).\vspace{\baselineskip}\\
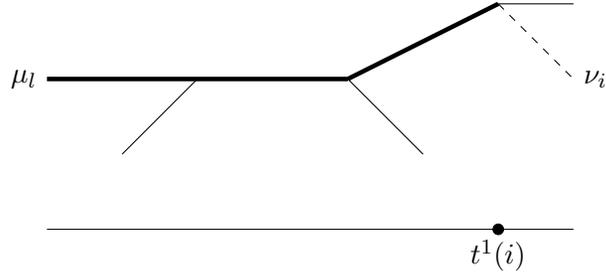
\begin{figure}[ht]
\begin{center}
\begin{tikzpicture}
\draw [line width=0.6mm] (0,0)--(4,0);
\draw [line width=0.6mm] (4,0)--(6,1);
\draw (6,1)--(7,1);
\draw [dashed] (6,1)--(7,0);
\draw (2,0)--(1,-1);
\draw (4,0)--(5,-1);
\draw (0,-2)--(7,-2);
\draw (0,0) node[anchor=east] {$\mu_l$};
\draw (7,0) node[anchor=west] {$\nu_i$};
\draw[fill] (6,-2) circle (2pt) node[anchor=north] {$t^1(i)$};
\end{tikzpicture}
\caption{Schematic drawing of end condition $(2,l,c,i)$.}
\label{fig:case3}
\end{center}
\end{figure}
\item We say $\eta$ satisfies end condition $(3,i)$ if 
\begin{itemize}
\item $t^1(i)>k$
\end{itemize}
In monodromy graph language, this corresponds to the following picture: The out-end corresponding to $\nu_i$ emanates from a vertex whose position is greater than $k$.
\end{enumerate}
Now we define start conditions.
\begin{enumerate}
\item We say $\eta$ satisfies start condition $(1,l)$ if 
\begin{itemize}
\item $t^2(l)\le k$
\item $s_{t^2(l)}\ge\sum_{j=1}^{l}\mu_j+1$
\end{itemize}
In monodromy graph language, this corresponds to the following picture: The in-end corresponding to $\mu_l$ is coloured dashed and joined with a chain of bold edges emanating at $\mu_{l'}$ for $l'>l$. (See \cref{fig:case2}).
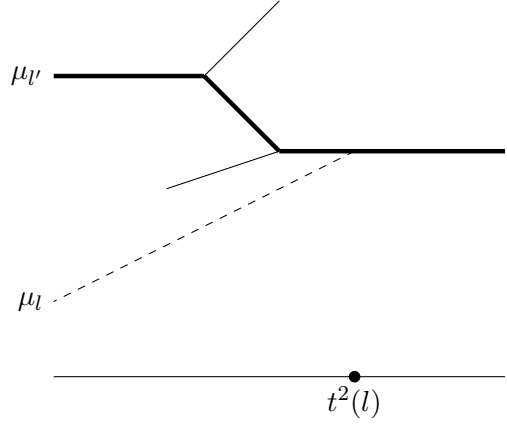
\begin{figure}[ht]
\begin{center}
\begin{tikzpicture}
\draw [line width=0.6mm] (0,0)--(2,0);
\draw (2,0)--(3,1);
\draw [line width=0.6mm] (2,0)--(3,-1);
\draw (3,-1)--(1.5,-1.5);
\draw [line width=0.6mm] (3,-1)--(6,-1);
\draw [dashed] (4,-1)--(0,-3);
\draw (0,0) node[anchor=east] {$\mu_{l'}$};
\draw (0,-3) node[anchor=east] {$\mu_{l}$};
\draw (0,-4)--(6,-4);
\draw[fill] (4,-4)circle (2pt) node[anchor=north] {$t^{2}(l)$};
\end{tikzpicture}
   \caption{Schematic drawing of end condition $(1,l)$ for $l'>l$.}
   \label{fig:case2}
   \end{center}
\end{figure}
\item We say $\eta$ satisfies start condition $(2,l,c)$ if 
\begin{itemize}
\item $t^2(l)\le k$
\item $s_{t^2(l)}\ge\sum_{j=1}^{l-1}\mu_j+c$ for $0\le c<\mu_l$
\end{itemize}
In monodromy graph language, this corresponds to one of the following two pictures: Either $\mu_l$ is joined with a chain of bold edges emanating at $\mu_m$ for $m>l$ or there is a chain of bold edges starting at $\mu_l$ and the first counter is greater or equal to $c$.
\begin{figure}[ht]
\begin{center}
\begin{tikzpicture}
\draw [line width=0.6mm] (0,0)--(2,0);
\draw (2,0)--(3,1);
\draw [line width=0.6mm] (2,0)--(3,-1);
\draw (3,-1)--(1.5,-1.5);
\draw [line width=0.6mm] (3,-1)--(6,-1);
\draw [dashed] (4,-1)--(0,-3);
\draw (0,0) node[anchor=east] {$\mu_{l'}$};
\draw (0,-3) node[anchor=east] {$\mu_{l}$};
\draw (0,-4)--(6,-4);
\draw[fill] (4,-4)circle (2pt) node[anchor=north] {$t^{2}(l)$};
\end{tikzpicture}
\begin{tikzpicture}
\draw [line width=0.6mm] (0,0)--(2,0);
\draw (2,0)--(3,1);
\draw [line width=0.6mm] (2,0)--(3,-1) node[midway, anchor=south west] {$c'$};
\draw (3,-1)--(1.5,-1.5);
\draw [line width=0.6mm] (3,-1)--(6,-1);
\draw (0,0) node[anchor=east] {$\mu_{l'}$};
\draw (0,-3)--(6,-3);
\draw[fill] (2,-3)circle (2pt) node[anchor=north] {$t^{2}(l)$};
\end{tikzpicture}
    \caption{Schematic drawing of the two possible graphs for end condition $(2,l,c)$: $l'>l$ in the left graph;$c'\ge c$ in the right graph.}
   \label{fig:case4}
   \end{center}
\end{figure}
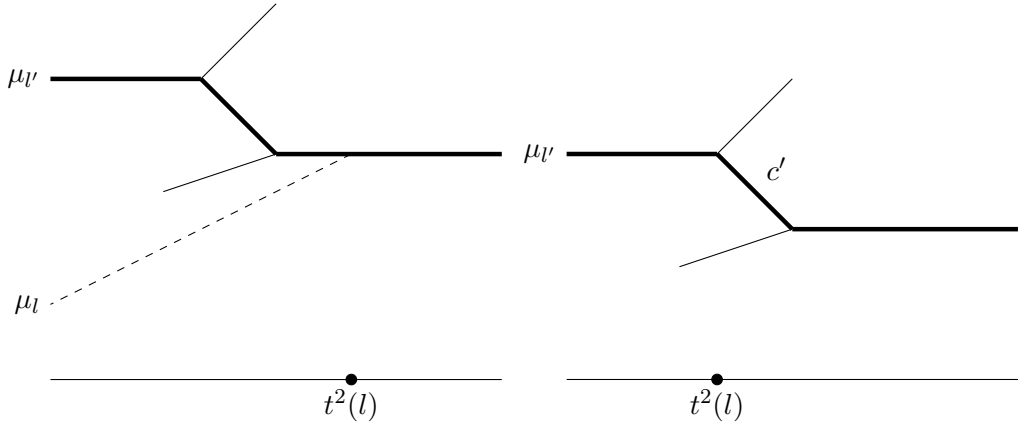
\item We say $\eta$ satisfies start condition $(3,l)$ if 
\begin{itemize}
\item $t^2(l)>k$
\end{itemize} 
In monodromy graph language, this corresponds to the following picture: The in-end corresponding to $\mu_l$ is adjacent to a vertex whose position is greater than $k$.
\end{enumerate}
\end{Definition}

\begin{Definition}
\label{def:general}
Let $S$ be a set of start conditions and $E$ a set of end conditions, i.e. for each $i\in[\ell(\mu)]$ (resp. $j\in[\ell(\nu)]$), there exists at most one tuple $(1,l)$, $(2,l,c)$ or $(3,l)$ (resp. $(1,l,i)$, $(2,l,c,i)$ or $(3,i)$) for some $l\in\{1,\dots,\ell(\mu)\},c\in\{0,\dots,\mu_l-1\}$ (resp. $i\in[\ell(\nu)],l\in[\ell(\mu],c\in\{0,\dots,\mu_l-1\}$), which is contained in $S$ (resp. $E$). Then we define $M_{p,0,r}^{\le,<,(2)}(\mu,\nu,S,E)$ to be the number of all triply interpolated factorisations $(\sigma_1,\tau_1,\dots,\tau_m,\sigma_2)$ of type $(0,\mu,\nu,p,0,r)$ with $\sigma_1$ as in \cref{equ:per} satisfying the conditions in $E$ and $S$.
\end{Definition}

\begin{Remark}
For $E=S=\emptyset$, we obtain triply interpolated Hurwitz numbers. Moreover, our methods from \cref{sec:pi0} can be applied to this generalised version to obtain piecewise polynomiality in the entries of $\mu$ and $\nu$ and the information in $S$ and $E$ with chambers given by $\mathcal{W}$.\\
By the same arguments as in \cref{sec:pi0}, $M_{p,0,r}^{\le,<,(2)}(\mu,\nu,S,E)$ may be expressed as a polynomial in the entries of $\mu$ and $\nu$ in each chamber induced by the hyperplane arrangement $\mathcal{W}$. We denote the polynomial in the chamber $C$ by $m_{p,0,r}^{\le,<,(2)}(\mu,\nu,S,E)(C)$.
\end{Remark}

Before we are ready to state the main theorem of this section, we introduce some Notation.

\begin{Notation}
Let $\mu$ be an ordered partition, $i\in\{1,\dots,\ell(\mu)\}$ and $\delta$ an integer, then we define the partition $(\mu,\delta)_i=\left(\mu_1,\dots,\mu_{i-1},\delta,\mu_i,\dots,\mu_{\ell(\mu)}\right)$.\\
Moreover, let $S$ be a set of start conditions and $I\subset\{1,\dots,\ell(\mu)\}$, then $S_I$ is the set of all start conditions $(1,l),(2,l,c)$ oder $(3,l)$ with $l\in I$.
\end{Notation}

\begin{figure}[ht]
\begin{tikzpicture}
\draw [line width=0.6mm] (0,2)--(1,2);
\draw (1,2)--(2,3);
\draw [line width=0.6mm] (1,2)--(2,1);
\draw [dashed] (2,1)--(2.5,1.5);
\draw (2,1)--(3,1) node[midway,anchor=north east] {$\delta$};
\draw (3,1)--(4,1) node[anchor=south] {$\omega(e_2)$};
\draw [line width=0.6mm] (0,0)--(2.25,0);
\draw (2.25,0)--(2.5,-0.5);
\draw [line width=0.6mm] (2.25,0)--(3,1) node[midway,anchor=north west] {$\omega(e_1)$};
\end{tikzpicture}\hspace{15pt}\begin{tikzpicture}
\draw [line width=0.6mm] (0,2)--(1,2);
\draw (1,2)--(2,3);
\draw [line width=0.6mm] (1,2)--(2,1);
\draw (2,1)--(2.5,1.5);
\draw [dashed] (2,1)--(3,1) node[midway,anchor=north east] {$\delta$};
\draw (3,1)--(4,1) node[anchor=south] {$\omega(e_2)$};
\draw [line width=0.6mm] (0,0)--(2.25,0);
\draw (2.25,0)--(2.5,-0.5);
\draw [line width=0.6mm] (2.25,0)--(3,1) node[midway,anchor=north west] {$\omega(e_1)$};
\end{tikzpicture}
\caption{The case (1a) and (2a) simultaneously on the left, the case (1b) and (2a) simultaneously on the right.}
\label{fig:local}
\end{figure}
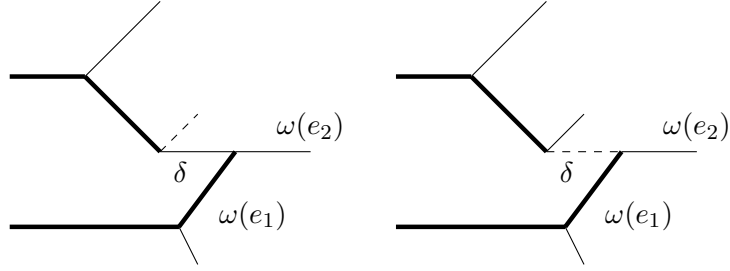

\begin{thm}
\label{thm:wall}
Let $\mu$ and $\nu$ be ordered partitions, $p,q,r$ non-negative integers (yielding genus $0$) and $C_1,C_2$ chambers separated by the wall defined by $\delta=\sum_{i\in I}\mu_i-\sum_{j\in J}\nu_j$ for subsets $I\subset\{1,\dots,\ell(\mu)\},J\subset\{1,\dots,\ell(\nu)\}$. Then
\begin{equation}W_{C_1}^{C_2}(\mu,\nu,k)=m_{p,0,r;\mu,\nu}^{\le,<,(2)}(\mu,\nu,S,E)(C_1)-m_{p,0,r;\mu,\nu}^{\le,<,(2)}(\mu,\nu,S,E)(C_2)
\end{equation} is a sum of products with factors of $\delta$ and two or three polynomials $m_{p',0,r'}^{\le,<,(2)}(\mu',\nu',S',E')(C_1)$ (resp. $m_{p',0,r'}^{\le,<,(2)}(\mu',\nu',S',E')(C_2)$), where the apostrophes indicate smaller input data. More precisely, the data $p',0,r',\mu',\nu',S',E'$ has to satisfy $p'<p$, $r'<r$, $\mu'=(\mu_I,\delta)_i$ (for some $i\in I$), $\nu'=(\nu_J,\delta)_j$ (for some $j\in J$), $S'$ is the union of $S_I$ and a start condition corresponding to the entry $\delta$ and $E'$ the union of $E_J$ and an end condition corresponding to the entry $\delta$ for $m_0^{k'}(\mu',\nu',S',E')(C_1)$ (resp. replacing $\delta$ by $-\delta$ for $m_0^{k'}(\mu',\nu',S',E')(C_2)$).
\end{thm}

\begin{Remark}
The same theorem is true for $p=0$ and arbitrary $q$ where we must make the according adjustments in the Definition of start and end conditions (i.e. change inequalities for counters to strict inequalities for counters in the regluing process).
\end{Remark}

\begin{figure}
\begin{tikzpicture}
\draw [line width=0.6mm] (0,2)--(1,2);
\draw (1,2)--(2,3);
\draw [line width=0.6mm] (1,2)--(2,1);
\draw [dashed] (2,1)--(2.5,1.5);
\draw (2,1)--(3,1) node[midway,anchor=north east] {$\delta$};
\draw [line width=0.6mm] (3,1)--(4,1) node[anchor=south east] {$\omega(e_2)$};
\draw [line width=0.6mm] (4,1)--(5,1);
\draw (4,1)--(5,2);
\draw [line width=0.6mm] (0,0)--(2.25,0);
\draw (2.25,0)--(2.5,-0.5);
\draw [line width=0.6mm] (2.25,0)--(3,1) node[midway,anchor=north west] {$\omega(e_1)$};
\end{tikzpicture}\hspace{10pt}\begin{tikzpicture}
\draw[->] (0,1.5)--(2,1.5);
\draw (0,0);
\end{tikzpicture}\hspace{10pt}
\begin{tikzpicture}
\draw [line width=0.6mm] (0,2)--(1,2);
\draw (1,2)--(2,3);
\draw [line width=0.6mm] (1,2)--(2,1);
\draw [dashed] (2,1)--(2.5,1.5);
\draw (2,1)--(3,1) node[midway,anchor=north east] {$\delta$};
\draw[interrupt] [line width=0.6mm] (3,1)--(4,1) node[anchor=south east] {$\omega(e_2)$};
\draw [line width=0.6mm] (4,1)--(5,1);
\draw (4,1)--(5,2);
\draw [line width=0.6mm] (0,0)--(2.25,0);
\draw (2.25,0)--(2.5,-0.5);
\draw[interrupt] [line width=0.6mm] (2.25,0)--(3,1) node[midway,anchor=north west] {$\omega(e_1)$};
\end{tikzpicture}

\begin{tikzpicture}
\draw [line width=0.6mm] (0,2)--(1,2);
\draw (1,2)--(2,3);
\draw [line width=0.6mm] (1,2)--(2,1);
\draw [dashed] (2,1)--(2.5,1.5);
\draw (2,1)--(3,1) node[midway,anchor=north east] {$\delta$};
\draw [line width=0.6mm] (3,1)--(4,1) node[anchor=south east] {$\omega(e_2)$};
\draw [line width=0.6mm] (4,1)--(5,1);
\draw (4,1)--(5,2);
\draw [line width=0.6mm] (0,0)--(2.25,0);
\draw (2.25,0)--(2.5,-0.5);
\draw [line width=0.6mm] (2.25,0)--(3,1) node[midway,anchor=north west] {$\omega(e_1)$};
\end{tikzpicture}\hspace{10pt}\begin{tikzpicture}
\draw[->] (0,1.5)--(2,1.5);
\draw (0,0);
\end{tikzpicture}\hspace{10pt}
\begin{tikzpicture}
\draw [line width=0.6mm] (0,2)--(1,2);
\draw (1,2)--(2,3);
\draw [line width=0.6mm] (1,2)--(2,1);
\draw [dashed] (2,1)--(2.5,1.5);
\draw[interrupt] (2,1)--(3,1) node[midway,anchor=north east] {$\delta$};
\draw [line width=0.6mm] (3,1)--(4,1) node[anchor=south east] {$\omega(e_2)$};
\draw [line width=0.6mm] (4,1)--(5,1);
\draw (4,1)--(5,2);
\draw [line width=0.6mm] (0,0)--(2.25,0);
\draw (2.25,0)--(2.5,-0.5);
\draw [line width=0.6mm] (2.25,0)--(3,1) node[midway,anchor=north west] {$\omega(e_1)$};
\end{tikzpicture}\hspace{10pt}
\caption{Schematic drawing of the cut-and-join process corresponding to cases (1a) and (2a) simultaneously.}
\label{fig:cutandjoin}
\end{figure}
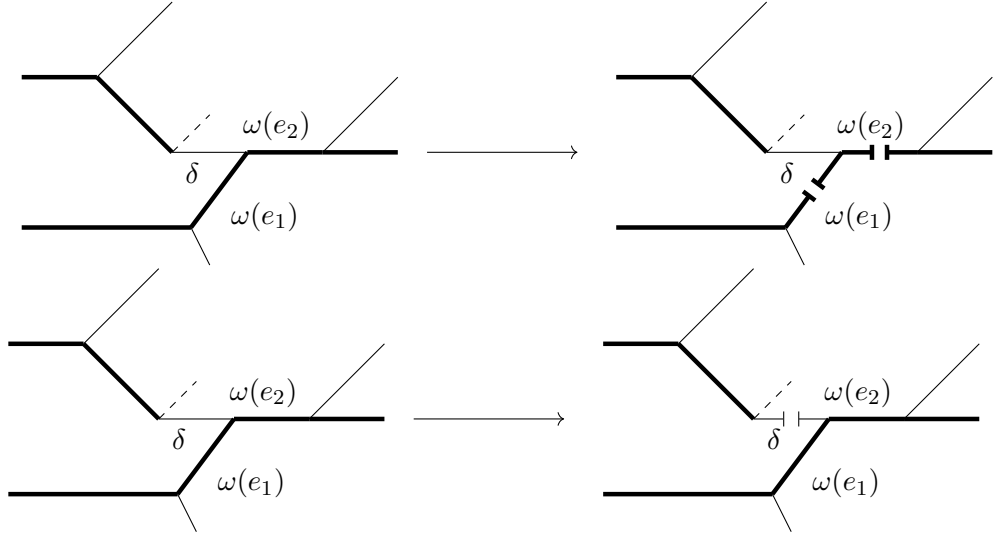

\begin{proof}[Proof of \cref{thm:wall}]
By the discussion at the beginning of this section, we already have an interpretation of the cutting and gluing process in terms of the polynomials $m_{p',0,r'}^{\le,<,(2)}(\mu',\nu',S',\allowbreak E')(C_1)$ and $m_{p',0,r'}^{\le,<,(2)}(\mu',\nu',S',E')(C_2)$. The proof is straightforward but involves many cases. We will work out the formula for (2a), since this is the most difficult case. The remaining cases can be worked out analogously.\\
As discussed before, we identify those graphs with different multiplicity in $C_1$ than in $C_2$. The following is the number of all graphs as in (2a), where $e_1$ (see \cref{equ:e1}) is not an edge adjacent to an in-end (and analogously in chamber $C_2$):
\begin{align}
&\sum_{\substack{I_1\subset I^c,J_1\subset J^c\\|I_1|>1\textrm{ or } J_1\neq\emptyset}}\sum_{p_1+p_2+p_3=k}\\
&\binom{m-k}{|I_1|+|J_1|-1-p_1,|I^c|+|J^c|-p_2,|I-I_1|+|J-J_1|-1-p_3}\\
&\sum_{t\in I_1}\sum_{1\le l\le l'\le \mu_t}\Bigg(m_{p_1,0,r_1}^{\le,<,(2)}(\mu_{I_1},(\nu_{J_1},\omega(e_1)),E_{J_1}\cup\{(2,t,l,|J_1|+1)\},S_{J_1})(C_1)\cdot\\
&m_{p_2,0,r_2}^{\le,<,(2)}((\mu_{I},\omega(e_1))_t,(\nu_{J},\omega(e_2),E_{J}\cup\{(2,t,l',|J|+1)\},S_{I}\cup\{(2,t,l)\})(C_1)\cdot\\&m_{p_3,0,r_3}^{\le,<,(2)}((\mu_{I^c-I_1},\omega(e_2))_t,\nu_{J^c-J_1},E_{I^c-I_1},S_{J^c-J_1}\cup\{(2,t,l')\})(C_1)\Bigg),
\end{align}
where $\omega(e_2)=\omega(e_1)+\delta)$ and we impose the condition on $r_i$ that we obtain genus $0$ in each factor. As mentioned before, by cutting along the two distinguished edges, we obtain three graphs of respective types as in \cref{equ:type1}, \cref{equ:type2} and \cref{equ:type3}. Each of these types corresponds to one of the three factors. The binomial coefficient counts the number of possible orderings on the vertices not affected by the monotonicity condition. All the other other cases work similarly, however what needs to be checked is that we obtain every graph exactly once. In fact, by the method above, we overcount, since (1a) and (2a) or (1b) and (2a) may appear simultaneously, which corresponds to the local picture illustrated in \cref{fig:local}: 
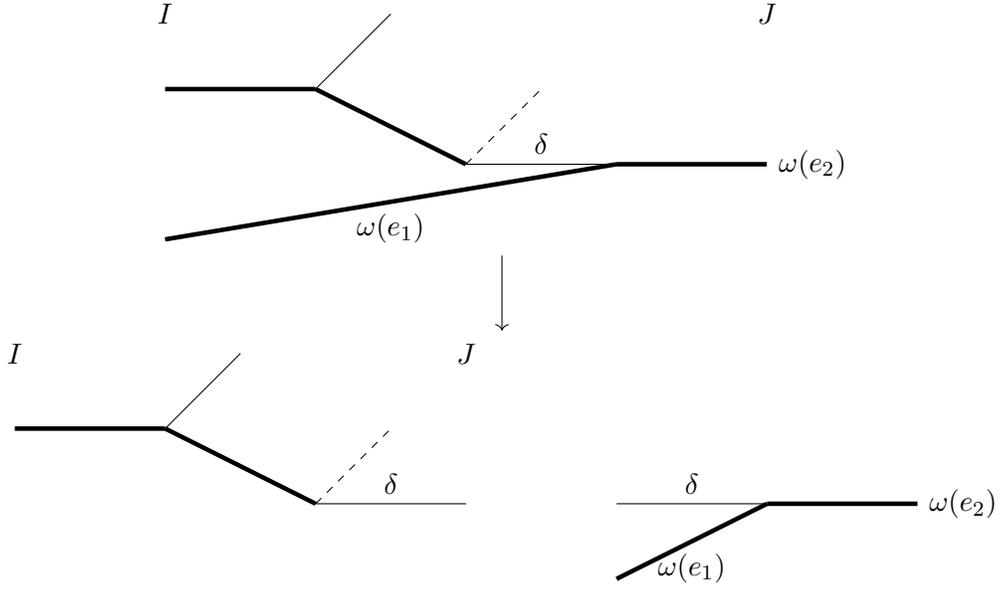
\begin{figure}[ht]
\begin{tikzpicture}
\draw (0,1) node {$I$};
\draw (8,1) node {$J$};
\draw [line width=0.6mm] (0,0)--(2,0);
\draw (2,0)--(3,1);
\draw [line width=0.6mm] (2,0)--(4,-1);
\draw[dashed] (4,-1)--(5,0);
\draw (4,-1)--(6,-1) node[anchor=south,midway] {$\delta$};
\draw [line width=0.6mm] (6,-1)--(8,-1) node[anchor=west] {$\omega(e_2)$};
\draw [line width=0.6mm] (0,-2)--(6,-1) node[midway,anchor=north] {$\omega(e_1)$};
\end{tikzpicture}

\begin{tikzpicture}
\draw[->] (0,0)--(0,-1);
\end{tikzpicture}

\begin{tikzpicture}
\draw (0,1) node {$I$};
\draw (6,1) node {$J$};
\draw [line width=0.6mm] (0,0)--(2,0);
\draw (2,0)--(3,1);
\draw [line width=0.6mm] (2,0)--(4,-1);
\draw[dashed] (4,-1)--(5,0);
\draw (4,-1)--(6,-1) node[anchor=south,midway] {$\delta$};
\draw (8,-1)--(10,-1) node[anchor=south,midway] {$\delta$};
\draw [line width=0.6mm] (10,-1)--(12,-1) node[anchor=west] {$\omega(e_2)$};
\draw [line width=0.6mm] (8,-2)--(10,-1) node[midway,anchor=north] {$\omega(e_1)$};
\end{tikzpicture}
\caption{Upper graph: Schematic drawing the graph corresponding to the correction term with set of in-ends indexed by $I\cup\{e_1\}$ and set of out-ends indexed by $J\cup\{e_2\}$. Lower graph: Schematic drawing of the graph obtained by cutting along $\delta$.}
\label{fig:correc}
\end{figure}
The same graph $\Gamma$ can be reglued from pieces that we obtain in case (1a) and (2a). By cutting along $\delta$ we obtain case (1a), however cutting along $e_1$ and $e_2$, we obtain case (2a). This cut-and-join process is illustrated in \cref{fig:cutandjoin}. The upper picture is a schematic drawing of the cutting along the edges $e_1$ and $e_2$, when case (1a) and (2a) happen at the same time. To compute the correction term, we need to remove all graphs $G$ as shown in the left of the lower picture. This is done by cutting along $\delta$ counting all graphs $\tilde{G}$ with out-end $\delta$ as in the right hand side of the lower picture and realising that the multiplicity of $G$ is $\delta$ times the multiplicity of $\tilde{G}$. In terms of the formula this means that we need to subtract the number of graphs as in upper picture in \cref{fig:correc}
\begin{equation}
\sum_{\substack{h\in I:\\h< t}}\delta\cdot m_{p_2-1,0,r_2}^{\le,<,(2)}(\mu_I,(\nu_J,\delta),E_J\cup (1,h,|J|+1),S_I).
\end{equation}
from the factor
\begin{equation}m_{p_2,0,r_2}^{\le,<,(2)}((\mu_{I},\omega(e_1))_t,(\nu_{J},\omega(e_2),E_{J}\cup\{(2,t,l',|J|+1)\},S_{I}\cup\{(2,t,l)\})(C_1)
\end{equation}
in each summand.
By a similar argument, we see that when the cases (1b) and (2a) appear simultaneously, we need to subtract the following term from 
\begin{equation}
m_{p_2,0,r_2}^{\le,<,(2)}((\mu_{I},\omega(e_1))_t,(\nu_{J},\omega(e_2),E_{J}\cup\{(2,t,l',|J|+1)\},S_{I}\cup\{(2,t,l)\})(C_1)
\end{equation}
in each summand:
\begin{equation}
\sum_{\substack{h\in I:\\h<t}}\sum_{g=1}^{\mu_h}\delta\cdot m_{p_2-1,0,r_2}^{\le,<,(2)}(\mu_I,(\nu_J,\delta),E_J\cup (2,h,g,|J|+1),S_I).
\end{equation}
\end{proof}

\nocite{johnson2015}
\bibliographystyle{alpha}

\newcommand{\etalchar}[1]{$^{#1}$}

\end{document}